\documentclass{pset}
\begin{document}
\title{Combinatorial Bounds in Distal Structures}
\author{Aaron Anderson}
\maketitle

\begin{abstract}
We provide polynomial upper bounds for the minimal sizes of distal cell decompositions in several kinds of distal structures, particularly weakly $o$-minimal and $P$-minimal structures. The bound in general weakly $o$-minimal structures generalizes the vertical cell decomposition for semialgebraic sets, and the bounds for vector spaces in both $o$-minimal and $p$-adic cases are tight. We apply these bounds to Zarankiewicz's problem in distal structures.
\end{abstract}

\section{Introduction}
Some of the strongest tools in geometric combinatorics revolve around partitioning space. These techniques fall largely into two categories, the polynomial partitioning method developed by Guth and Katz \cite{gk}, and versions of the cutting lemma for various cell decompositions \cite{chaz91}. 
While the polynomial method has yielded impressive results, its reliance on B\'ezout's Theorem limits its scope to questions about algebraic and semialgebraic sets. If one tries to generalize it to sets definable in $o$-minimal structures other than real closed fields, B\'ezout's theorem can fail \cite{bezout}. The cutting lemma method, however, can be generalized to more complicated sets using the language of model theory. Distal cell decompositions, defined in \cite{cgs}, provide an analogous definition to the stratification or vertical cell decomposition results known for $\R$, with a similar cutting lemma, for families of sets definable in a suitable first-order structure, known as a distal structure.

We then study distal cell decompositions through the lens of shatter functions. In \cite{vcdensityi}, the \emph{dual shatter function} $\pi^*_\Phi$ of a set $\Phi$ of formulas is defined so that $\pi^*_\Phi(n)$ is the maximum cardinality of the set of $\Phi$-types over a parameter set of size $n$. We define an analogous shatter function $\pi_\mathcal{T}(n)$ for each distal cell decomposition $\mathcal{T}$, where instead of counting all $\Phi$-types, we count the maximum number of cells needed for a distal cell decomposition against $n$ sets (See Definition \ref{distal_density}). This shatter function grows polynomially in a distal structure, so each $\mathcal{T}$ has some exponent $t \in \mathbb{R}$ such that $\pi_\mathcal{T}(n)=\bigO(n^t)$. This exponent is what determines the effectiveness of the cutting lemma for combinatorial applications. Just as the dual VC density of $\Phi$ is defined to be the rate of growth of $\pi^*_\Phi$, we define the \emph{distal density} of $\Phi$ to be the infimum of the exponents of all distal cell decompositions $\mathcal{T}$ for $\Phi$.

In this article, we construct and bound the sizes of distal cell decompositions for definable families in several distal structures, namely the weakly $o$-minimal structures, including a better bound on ordered vector spaces, the field $\Q_p$, and its linear reduct. Then we apply these bounds to some combinatorial problems.

\subsection{Main Results}
Our first theorem constructs distal cell decompositions (see Definition \ref{def_dcd}) for all sets of formulas $\Phi(x;y)$, with $x$ and $y$ tuples of variables of arbitrary finite length, in some structure $\mathcal{M}$, given a distal cell decomposition for all sets of formulas $\Phi(x;y)$, with with $\abs{x}=1$. This construction by inducting on the dimension generalizes the stratification result in \cite{chaz91}, which essentially constructs distal cell decompositions for $\R$ as an ordered field. It is also similar to Theorem 7.1 in \cite{vcdensityi}, which provides an analogous bound for the VC density of a set of formulas in many dimensions assuming the strong VCd property in dimension 1.

\begin{thm*}[Theorem \ref{dimensioninduction}]
	Let $\mathcal{M}$ be a structure in which all finite sets $\Phi(x;y)$ of formulas with $\abs{x}=1$ admit a distal cell decomposition with $k$ parameters (see Definition \ref{def_param_exp}), and for some $d_0\in \N$, all finite sets $\Phi(x;y)$ of formulas with $\abs{x}=d_0$ admit distal cell decompositions of exponent at most $r$. Then all finite sets $\Phi(x;y)$ of formulas with $\abs{x}=d\geq d_0$ admit distal cell decompositions of exponent $k(d-d_0)+r$.
\end{thm*}

In sections \ref{omin}, \ref{pres}, \ref{qpaff}, and \ref{qpmac}, we prove upper bounds on the exponents of distal cell decompositions in weakly $o$-minimal structures, as well as the field $\Q_p$ and its linear reduct.
Those results are summarized and contrasted with the best-known bounds for the dual VC density, in the following theorem:
\begin{thm}\label{thm_table}
	Let $\mathcal{M}$ be a structure from the first column of this table.
	Then any formula $\phi(x;y)$ has dual VC density bounded by the corresponding value in the second column, and admits a distal cell decomposition with exponent
	bounded by the value in the third column. Thus also its distal density is bounded by the value in the third column.

	\begin{figure}[H]
		\begin{tabular}{|c | c | c|}
		\hline $\mathcal{M}$& Dual VC density & Distal Density \\ \hline
		 $o$-minimal expansions of groups & $\abs{x}$ & $2\abs{x}-2$ (1 if $\abs{x}=1$)  \\ \hline
		 weakly $o$-minimal structures& $\abs{x}$ & $2\abs{x}-1$\\ \hline
		 ordered vector spaces over ordered division rings & $\abs{x}$ & $\abs{x}$\\ \hline
		 Presburger arithmetic & $\abs{x}$ & $\abs{x}$\\ \hline
		 $\Q_p$ the valued field & $2\abs{x}-1$ & $3\abs{x}-2$\\ \hline
		 $\Q_p$ in the linear reduct & $\abs{x}$ & $\abs{x}$ \\ \hline
	\end{tabular}
	\end{figure}
\end{thm}
\begin{proof}
The Dual VC density bounds are from \cite{vcdensityi}, except for the bound for the linear reduct of $\Q_p$, which is from \cite{bobkov}.

Theorem \ref{ominthm} establishes the bound for weakly $o$-minimal structures by constructing a distal cell decomposition in the 1-dimensional case, and then applying Theorem \ref{dimensioninduction}. Taking into account \cite{cgs}, we improve that bound for $o$-minimal expansions of fields to match the bound from \cite{chaz91} for the case of $\R$ as an ordered field. This improves \cite[Theorem 4.0.9]{barone}, which provides a cell decomposition with $\bigO(\abs{B}^{2\abs{x}-1})$ uniformly definable cells for $\mathcal{M}$ an $o$-minimal expansion of a real closed field.

Theorem \ref{ominvec} shows that the distal density of any finite set of formulas $\Phi(x;y)$ in an ordered vector space over an ordered division ring matches the VC density. In particular, the distal exponent of $\Phi$ is bounded by $\abs{x}$, which is optimal. This also works for any $o$-minimal locally modular expansion of an abelian group, and Theorem \ref{pres} shows the same results for $\Z$ in Presburger's language.

Theorem \ref{qpaffthm} shows that the distal density matches the VC density for any finite set of formulas $\Phi(x;y)$ in $\Q_p$ equipped with its reduced linear structure in the language $\mathcal{L}_\mathrm{aff}$ described by Leenknegt in \cite{lee}. The proof adapts Bobkov's bound on VC density in the same structure \cite{bobkov}.

Theorem \ref{qpmacthm} establishes the bound for $\Q_p$ or any other $P$-minimal field with quantifier-elimination and definable Skolem functions in Macintyre's language by constructing a distal cell decomposition in the 1-dimensional case and applying Theorem \ref{dimensioninduction}.
\end{proof}

Finally in Section \ref{zarsec} we apply these results to combinatorics. We combine them with the results on Zarankiewicz's problem from \cite{cgs} to prove a bound on the number of edges in bipartite graphs definable in distal structures which omit some (oriented) complete bipartite graph $K_{s,u}$, similar to the bound given by Theorem 1.2 from \cite{sheff}.

\begin{cor*}[Corollary \ref{zarcor}, expressed in terms of distal density]
	Let $\mathcal{M}$ be a structure and $t \in \N_{\geq 2}$. Assume that $E(x,y)\subseteq M^{\abs{x}}\times M^{\abs{y}}$ is a definable relation given by an instance of a formula $\theta(x,y;z)\in\mathcal{L}$, such that the formula $\theta'(x;y,z):=\theta(x,y;z)$ has distal density at most $t$, and the graph $E(x,y)$ does not contain $K_{s,u}$. Then for every $\varepsilon \in \R_{>0},$ there is a constant $\alpha=\alpha(\theta,s,u, \varepsilon)$ satisfying the following.

	For any finite $P\subseteq M^{\abs{x}},Q \subseteq M^{\abs{y}}$, $\abs{P}=m,\abs{Q}=n$, we have:
	$$\abs{E(P,Q)}\leq \alpha\left(m^{\frac{(t-1)s}{ts-1}+\varepsilon}n^{\frac{t(s-1)}{ts-1}}+m+n\right).$$
\end{cor*}

This corollary then lets us place bounds on graphs in the following contexts:

\begin{cor*}[Corollary \ref{cor:Rexp}]
	Assume that $E(x,y)\subseteq \R^{\abs{x}}\times \R^{\abs{y}}$ is a relation given by a boolean combination of exponential-polynomial (in)equalities, and the graph $E(x,y)$ does not contain $K_{s,u}$. Then there is a constant $\alpha=\alpha(\theta,s,u)$ satisfying the following.

	For any finite $P\subseteq \R^{\abs{x}},Q \subseteq \R^{\abs{y}}$, $\abs{P}=m,\abs{Q}=n$, we have:
	$$\abs{E(P,Q)}\leq \alpha\left(m^{\frac{(2\abs{x}-2)s}{(2\abs{x}-1)s-1}}n^{\frac{(2\abs{x}-1)(s-1)}{(2\abs{x}-1)s-1}+\varepsilon}+m+n\right).$$
\end{cor*}

(Here an \emph{exponential-polynomial (in)equality} is an (in)equality between functions $\R^n \to \R$ in $\Z[x_1,\dots,x_n,e^{x_1},\dots,e^{x_n}]$ as in \cite{exppoly}.

\begin{cor*}[Corollary \ref{cor:Qpan}]
	Assume that $E(x,y)\subseteq \Z_p^{\abs{x}}\times \Z_p^{\abs{y}}$ is a subanalytic relation, and the graph $E(x,y)$ does not contain $K_{s,u}$. Then there is a constant $\alpha=\alpha(\theta,s,u)$ satisfying the following.

	For any finite $P\subseteq \Z_p^{\abs{x}},Q \subseteq \Z_p^{\abs{y}}$, $\abs{P}=m,\abs{Q}=n$, we have:
	$$\abs{E(P,Q)}\leq \alpha\left(m^{\frac{(3\abs{x}-3)s}{(3\abs{x}-2)s-1}}n^{\frac{(3\abs{x}-2)(s-1)}{(3\abs{x}-2)s-1}+\varepsilon}+m+n\right).$$
\end{cor*}

Here subanalytic relations are defined in the sense of \cite{subanalytic}.

\subsection*{Acknowledgements} 
We thank Artem Chernikov for excellent guidance, and Alex Mennen and Pierre Touchard for useful conversations. This paper was written with support from the NSF CAREER grant DMS-1651321.

\section{Preliminaries}
In this section, we review the notation and model-theoretic framework necessary to understand distal cell decompositions. For further background on these definitions, see \cite{cs} and \cite{cgs}.

Firstly, we review asymptotic notation:
\begin{defn}
	Let $f, g : \N \to \R_{\geq 0}$.
	\begin{itemize}
		\item We will say $f(x) = \bigO(g(x))$ to indicate that there exists $C \in \R_{>0}$ such that for $n \in \N_{>0}$, $f(n) \leq Cg(x)$.
		\item We will say $f(x) = \Omega(g(x))$ to indicate that there exists $C \in \R_{>0}$ such that for $n \in \N_{>0}$, $f(n) \geq Cg(x)$.
	\end{itemize}

	If $f,g : \N \times \N \to \R_{\geq 0}$, then $f(x,y) = \bigO(g(x,y))$ indicates that there is a constant $C \in \R_{>0}$ such that for all $m,n \in \N_{>0}$, $f(m,n) \leq Cg(m,n)$.
\end{defn}

Throughout this section, let $\mathcal{M}$ be a first-order structure in the language $\mathcal{L}$. We will frequently refer to $\Phi(x;y)$ as a set of formulas, which will implicitly be in the language $\mathcal{L}$. Each formula in $\Phi$ will have the same variables, split into a tuple $x$ and a tuple $y$, where, for instance, $\abs{x}$ represents the length of the tuple $x$. We use $M$ to refer to the universe, or underlying set, of $\mathcal{M}$, and $M^n$ to refer to its $n$th Cartesian power. If $\phi(x;y)$ is a formula with its variables partitioned into $x$ and $y$, and $b \in M^{\abs{y}}$, then $\phi(M^{\abs{x}};b)$ refers to the definable set $\{a \in M^{\abs{x}}:\mathcal{M}\models \phi(a,b)\}$. We also define the dual formula of $\phi(x;y)$ to be $\phi^*(y;x)$ such that $\mathcal{M}\models \A x\A y\phi(x;y)\leftrightarrow \phi^*(y;x)$, and similarly define $\Phi^*(y;x)$ to be the set $\{\phi^*(y;x):\phi(x;y) \in \Phi(x;y)\}$.

\begin{defn}
	For sets $A,X \subseteq M^d$, we say that $A$ \emph{crosses} $X$ if both $X \cap A$ and $X\cap \neg A$ are nonempty.
\end{defn}

\begin{defn}
	Let $B \subseteq M^t$.
	\begin{itemize}
		\item For $\phi(x;y)$ with $\abs{y}=t$, we say that $\phi(x;B)$ \emph{crosses} $X \subseteq M^{\abs{x}}$ when there is some $b\in B$ such that $\phi(M^{\abs{x}};b)$ crosses $X$.
		\item For $\Phi(x;y)$ with $\abs{y}=t$, we say that $X \subseteq M^{\abs{x}}$ \emph{is crossed by} $\Phi(x;B)$ when there is some $\phi \in B$ such that $\phi(x;B)$ crosses $X$.
	\end{itemize}
\end{defn}

\begin{defn}
	We define $S^\Phi(B)$ to be the set of complete $\Phi$-types over a set $B\subseteq M^{\abs{y}}$ of parameters, or alternately, the set of maximal consistent subsets of $\{\varphi(x;b):\varphi \in \Phi,b \in B\}\cup \{\neg \varphi(x;b):\varphi \in \Phi,b \in B\}$.
\end{defn}

Throughout this article, we will want to use the concepts of VC density and dual VC density.
\begin{defn}
	Let $\Phi(x;y)$ be a finite set of formulas. 
	\begin{itemize}
		\item For $B \subseteq M^{\abs{y}}$, define $\pi^*_\Phi(B) := \abs{S^\Phi(B)}$.
		\item For $n \in \N$, define $\pi^*_\Phi(n):=\max_{B \subseteq M^{\abs{y}},\abs{B}=n}\pi^*_\Phi(B)$.
		\item Define the \emph{dual VC density} of $\Phi$, $\mathrm{vc}^*(\Phi)$, to be the infimum of all $r \in \R_{> 0}$ such that there exists $C \in \R$ with $\abs{S^{\Phi}(B)}\leq C\abs{B}^r$ for all choices of $B$. Equivalently, we can define $\mathrm{vc}^*(\Phi)$ to be
		$$\limsup_{n\to\infty}\frac{\log \pi^*_\Phi(n)}{\log n}.$$
		\item Dually, we define $\pi_\Phi:=\pi^*_{(\Phi^*)}$ and define the \emph{VC density} of $\Phi$ to be $\mathrm{vc}(\Phi)=\mathrm{vc}^*(\Phi^*)$.
	\end{itemize}
\end{defn}
This definition of (dual) VC density of sets of formulas comes from Section 3.4 of \cite{vcdensityi}, which relates it to the other definitions of VC density.

\begin{defn}
	An \emph{abstract cell decomposition} for $\Phi(x;y)$ is a function $\mathcal{T}$ that assigns to each finite $B \subset M^{\abs{y}}$ a set $\mathcal{T}(B)$ whose elements, called cells, are subsets of $M^{\abs{x}}$ not crossed by $\Phi(x;B)$, and cover $M^{\abs{x}}$ so that $M^{\abs{x}}=\bigcup\mathcal{T}(B)$.
\end{defn}

\begin{eg}
	Fix $\Phi(x;y)$. For each type $p(x) \in S^\Phi(B)$, the set $p(M^{\abs{x}})$ is a definable subset of $M^{\abs{x}}$, as $p(x)$ is equivalent to a boolean combination of formulas $\phi(x;b)$ for $\phi \in \Phi$ and $b \in B$. Define $\mathcal{T}_{\mathrm{vc}}(B):=\{p(M^{\abs{x}}):p\in S^\Phi(B)\}$. Then $\mathcal{T}_{\mathrm{vc}}$ is an abstract cell decomposition with $\abs{\mathcal{T}_{\mathrm{vc}}(B)}=\abs{S^{\Phi}(B)} = \pi^*_\Phi(B)$.
\end{eg}

\begin{prop} \label{abstractsize}
	For any abstract cell decomposition $\mathcal{T}$ of $\Phi(x;y)$ and any finite $B\subseteq M^{\abs{y}}$, $\abs{\mathcal{T}(B)}\geq \pi^*_\Phi(B)$.
\end{prop}
\begin{proof}
	As each cell $\Delta \in \mathcal{T}(B)$ is not crossed by $\Phi(x;B)$, its elements must all have the same $\Phi$-types over $B$. Thus there is a function $f:\mathcal{T}(B)\to S^\Phi(B)$ sending each cell to the $\Phi$-type over $B$ of its elements. Each type in $S^\Phi(B)$ is consistent and definable by a formula, and thus must be realized in $M$, so there must be at least one cell of $\mathcal{T}(B)$ containing formulas of that type. Thus $f$ is a surjection, and $\abs{\mathcal{T}(B)}\geq \abs{S^\Phi(B)}$.
\end{proof}

\begin{defn}\label{def_dcd}
	Let $\Phi(x;y)$ be a finite set of formulas without parameters. Then a \emph{distal cell decomposition} $\mathcal{T}$ for $\Phi$ is an abstract cell decomposition defined using the following data:
	\begin{itemize}
		\item A finite set $\Psi(x;y_1,\dots,y_k)$ of formulas (without parameters) where $\abs{y_1}=\dots=\abs{y_k}=\abs{y}$.
		\item For each $\psi \in \Psi$, a formula (without parameters) $\theta_\psi(y;y_1,\dots,y_k)$.
	\end{itemize}

	Given a finite set $B \subseteq M^{\abs{y}}$, let $\Psi(B):=\{\psi(M^{\abs{x}};b_1,\dots,b_k):\psi \in \Psi,b_1,\dots,b_k\in B\}$. This is the set of potential cells from which the cells of the decomposition are chosen. Then for each potential cell $\Delta =\psi(M^{\abs{x}};b_1,\dots,b_k)$, we let $\mathcal{I}(\Delta)=\theta_\psi(M^{\abs{y}};b_1,\dots,b_k)$. Then we define $\mathcal{T}(B)$ by choosing the cells $\Delta \in \Psi(B)$ such that $B \cap \mathcal{I}(\Delta)=\emptyset,$ that is, $\mathcal{T}(B)=\{\Delta \in \Psi(B):B\cap \mathcal{I}(\Delta)=\emptyset\}$.
\end{defn}

In the rest of this article, when $\Phi(x;y)$ is a finite set of formulas, we will assume that $\Phi$ is defined without parameters. 

The following lemma will be useful in defining distal cell decompositions later on:
\begin{lem}\label{bool_comb}
	Let $\Phi(x;y)$ be a finite set of formulas, and let $\Phi'(x;y)$ be a finite set of formulas such that each formula in $\Phi$ is a boolean combination of formulas in $\Phi'$. Then if $\mathcal{T}$ is a distal cell decomposition for $\Phi'$, it is also a distal cell decomposition for $\Phi$.
\end{lem}
\begin{proof}
	The definability requirements for a distal cell decomposition do not depend on the set of formulas $\Phi$, so it suffices to show that $\mathcal{T}$ is an abstract cell decomposition for $\Phi$, or that for a given $B$, each cell $\Delta \in \mathcal{T}(B)$ is not crossed by $\Phi(x;B)$. As for any $\varphi \in \Phi$, $b \in B$, $\varphi(x;b)$ is a boolean combination of formulas in $\Phi'(x;B)$, and all of these have a fixed truth value on $\Delta$, so does $\varphi(x;b)$.
\end{proof}

We now consider a few ways of counting the sizes of distal cell decompositions:
\begin{defn}\label{def_param_exp}
	Let $\mathcal{T}$ be a distal cell decomposition for the finite set of formulas $\Phi(x;y)$, whose cells are defined by formulas in the set $\Psi$.
	\begin{itemize}
	 	\item We say that $\mathcal{T}$ \emph{has $k$ parameters} if every formula in $\Psi$ is of the form $\psi(x;y_1,\dots,y_k)$.
	 	\item We say that $\mathcal{T}$ \emph{has exponent $r$} if $\abs{\mathcal{T}(B)}=\bigO(\abs{B}^r)$ for all finite $B \subseteq M^{\abs{y}}$.
	 \end{itemize}
\end{defn}

Note that even if $\mathcal{T}$ has $k$ parameters, not every formula $\psi$ used to define $\mathcal{T}$ needs to use all $k$ parameters. In practice, we will sometimes define distal cell decompositions using formulas with different numbers of variables, but as each distal cell decomposition is defined using finitely many formulas, we can just take $k$ to be the maximum number of parameters used by any one formula, and add implicit variables to the rest.

\begin{defn}\label{distal_density}
	Let $\Phi(x;y)$ be a finite set of formulas. Then define the \emph{distal density} of $\Phi$ to be the infimum of all reals $r\geq 0$ such that there exists a distal cell decomposition $\mathcal{T}$ of $\Phi$ of exponent $r$. If no $\mathcal{T}$ exists for $\Phi$, the distal density is defined to be $\infty$.
\end{defn}

\begin{problem}
	Note that if $\Phi$ has distal density $t$, it is not known if $\theta$ must have a distal cell decomposition of exponent precisely $t$.
\end{problem}

\begin{defn}
	We also define a shatter function $\pi_\mathcal{T}(n):=\max_{\abs{B}=n} \abs{\mathcal{T}(B)}$. The distal density of $\Phi$ can equivalently be defined as the infimum of
	$$\limsup_{n \to \infty}\frac{\log \pi_\mathcal{T}(n)}{\log n}$$
	over all distal cell decompositions $\mathcal{T}$ of $\Phi$, if any exist.
\end{defn}

\begin{prop}
	For any finite set of formulas $\Phi(x;y)$, $\pi_\mathcal{T}(n)\geq \pi^*_\Phi(n)$ for all $n \in \N$, and the distal density of $\Phi$ is at least $\mathrm{vc}^*(\Phi)$.
\end{prop}
\begin{proof}
	By Proposition \ref{abstractsize}, for every distal cell decomposition $\mathcal{T}$, $\abs{\mathcal{T}(B)}\geq \abs{S^\Phi(B)}$. Thus
	$$\mathrm{vc}^*(\Phi)\leq \limsup_{n \to \infty}\frac{\log \pi^*_\Phi(n)}{\log n}\leq \limsup_{n \to \infty}\frac{\log \pi_\mathcal{T}(n)}{\log n}$$
	so after taking the infimum over all $\mathcal{T}$, the distal density is at least $\mathrm{vc}^*(\Phi)$.
\end{proof}

Also, just by defining $\Phi(x;y)$ to be $\{x=y\}$, where $\abs{x}=\abs{y}=d$, we see that $\abs{S^\Phi(B)}\geq \abs{B}^d$, so we see that for every $d$, there is a $\Phi$ with both VC- and distal densities at least $d$ in any structure.

\begin{eg}
 	Chernikov, Galvin and Starchenko found that if $\mathcal{M}$ is an $o$-minimal expansion of a field, and $\abs{x}=2$, then any $\Phi(x;y)$ admits a distal cell decomposition with $\abs{\mathcal{T}(B)}=\bigO(\abs{B}^2)$ for all finite $B$ \cite{cgs}. Thus the distal density of such a $\Phi$ is at most 2.
\end{eg}

So far, we have defined distal cell decompositions and distal density in the context of a particular structure. In fact, if $\Phi(x;y)$ is a finite set of $\mathcal{L}$-formulas, and $T$ a complete $\mathcal{L}$-theory, we will show that the distal density of $\Phi(x;y)$ is the same in every model of $T$, so we can define \emph{the distal density of $\Phi$ over} $T$ to be the distal density of $\Phi$ in any model of $T$. (This uses the fact that the formulas in $\Phi$ and the formulas defining a distal cell decomposition are required to be parameter-free.)

\begin{prop}
Let $\Phi(x;y)$ be a finite set of $\mathcal{L}$-formulas, and $\mathcal{M} \equiv \mathcal{M}'$ be elementarily equivalent $\mathcal{L}$-structures. Then if $\Phi$ admits a distal cell decomposition $\mathcal{T}$ in $\mathcal{M}$, the same formulas define a distal cell decomposition for $\Phi$ in $\mathcal{M}'$. Thus we can refer to $\mathcal{T}$ as being a distal cell decomposition for $\Phi$ over the theory $T = \mathrm{Th}(\mathcal{M})$. Also, the shatter function $\pi_{\mathcal{T}}$, and thus the distal exponent of $\mathcal{T}$ and the distal density of $\Phi$, will be equal for $\mathcal{M}$ and $\mathcal{M}'$, and can be viewed as properties of the theory $T$.
\end{prop}
\begin{proof}
Let $\mathcal{T}$ be a distal cell decomposition for $\Phi$ over $\mathcal{M}$, consisting of a set $\Psi(x;y_1,\dots,y_k)$ of formulas, and a formula $\theta_\psi$ for each $\psi \in \Psi$ (as in Definition \ref{def_dcd}). Then to verify that the same formulas define a distal cell decomposition for $\Phi$ over $\mathcal{M}'$, we must simply check that for all finite $B \subset M'^{\abs{y}}$, the set of cells $\mathcal{T}(B)$ covers $M'^{\abs{x}}$, and that no cell of $\mathcal{T}(B)$ is crossed by $\Phi(x;B)$.

It is enough to show that these facts can be described with first-order sentences. Fix some natural number $n$, and we will find a first-order sentence that shows that for all $B = \{b_1,\dots, b_n\}$, the cells of $\mathcal{T}(B)$ cover the space and are not crossed. We can encode that the cells of $\mathcal{T}(B)$ cover $M'^{\abs{x}}$ with the sentence
$$\forall y_1,\dots, y_n, \forall x, \bigwedge_{\psi \in \Psi, i_1,\dots,i_k \in \{1,\dots,n\}}\psi(x;y_{i_1},\dots,y_{i_k}) \wedge \bigwedge_{i = 1}^n \neg \theta_\psi(y_i;y_{i_1},\dots,y_{i_k}).$$
When interpreted over $\mathcal{M}'$, this simply states that for any choice of $n$ parameters $b_1,\dots, b_n$ and any $x_0 \in M'^{\abs{x}}$, there is some $\psi, i_1,\dots,i_k$ such that $\psi(x;b_{i_1},\dots,b_{i_k})$ defines a valid cell, which contains $x_0$.
Similarly, to show that the cell defined by $\psi(x;b_{i_1},\dots,b_{i_k})$, if included in the cell decomposition, is not crossed by $\Phi(x;B)$, we can use the following sentence, showing that for all $B = \{b_1,\dots, b_n\}$, if for some $i$ and some $\varphi \in \Phi$, $\phi(x;b_i)$ crosses $\psi(x;b_{i_1},\dots,b_{i_k})$, then $\psi(x;b_{i_1},\dots,b_{i_k})$ is not a valid cell:
\begin{align*}
\forall y_1,\dots, y_n, \left(\bigvee_{\varphi \in \Phi, 1 \leq i \leq n} \exists x_1, x_2, \varphi(x_1;y_i) \wedge \neg \varphi(x_2;y_i) \wedge \psi(x_1;y_{i_1},\dots,y_{i_k}) \wedge \psi(x_2;y_{i_1},\dots,y_{i_k})\right) \\
\to \bigvee_{i = 1}^n \theta_\psi(y_i;y_{i_1},\dots,y_{i_k}).
\end{align*}

Now it suffices to show that the shatter function $\pi_{\mathcal{T}}$ is the same in both models, as the distal exponent of $\mathcal{T}$ and distal density of $\Phi$ are defined in terms of these shatter functions. 

To say that $\pi_\mathcal{T}(n) \leq m$ in $\mathcal{M}$ is to say that for all $b_1,\dots, b_n \in M^{\abs{y}}$, there are at most $m$ cells in $\mathcal{T}(B)$. This is the disjunction of a finite number of cases, which we will index by $A_1,\dots,A_m$, where each $A_i \subset \Psi \times \{1,\dots,n\}^k$, as each tuple $t = (\psi_t, t_1,\dots,t_n) \in \Psi \times \{1,\dots,n\}^k$ corresponds to a potential cell $\Delta_s = \psi_t(x;b_{t_1},\dots,b_{t_n})$. Then in the case indexed by $A_1,\dots, A_m$, there is a first-order sentence stating that for all $1 \leq i \leq n$ and $s,t \in A_i$, the formulas $\Delta_s$ and $\Delta_t$ are equivalent, and for all tuples $t = (\psi, i_1,\dots,i_n)$ not contained in any $A_i$, $t$ is not a valid cell, as implied by $\bigvee_{j =1}^n\theta_\psi(b_j;b_{i_1},\dots,b_{i_n})$. The disjunction of all these sentences states that there are at most $m$ distinct cells in $\mathcal{T}(\{b_1,\dots,b_n\})$, and if $b_1,\dots,b_n$ are replaced with universally-quantified variables, we find a sentence that states that $\pi_{\mathcal{T}}(n)\leq m$. Thus for all $n$, $\pi_{\mathcal{T}}(n)$ evaluates to the same number over any model of the theory of $\mathcal{M}$.
\end{proof}

Distality of a theory was defined originally in terms of indiscernible sequences in \cite{simon}. We will not present that definition here, but we will take the following equivalence as a definition:
\begin{fact}
	The following are equivalent for any first-order structure $\mathcal{M}$:
	\begin{enumerate}
		\item $\mathcal{M}$ is distal.
		\item For every formula $\phi(x;y)$, $\{\phi\}$ admits a distal cell decomposition.
		\item For every finite set of formulas $\Phi(x;y)$, $\Phi$ admits a distal cell decomposition.
	\end{enumerate}
\end{fact}
\begin{proof}
	The equivalence of (1) and (2) is by \cite{c-simon} (see \cite[Fact 2.9]{cgs} for a discussion). Clearly (3) implies (2), so it suffices to show that (2) implies (3).

	For a given $\Phi(x;y)$, assume each $\phi\in \Phi$ admits a distal cell decomposition $\mathcal{T}_\phi$. Then for finite $B\subseteq M^{\abs{y}}$, we define $\mathcal{T}(B)$ to consist of all nonempty intersections $\bigcap_{\phi \in \Phi}\Delta_\phi$, where each $\Delta_\phi$ is chosen from $\mathcal{T}_\phi(B)$. These cells will cover $M^{\abs{x}}$, as each $a \in M^{\abs{x}}$ belongs to some $\Delta_\phi$ for each $\phi$, and thus belongs to their intersection. Any cell $\Delta = \bigcap_{\phi \in \Phi}\Delta_\phi$ will not be crossed by $\Phi(x;B)$, as for each $\phi \in \Phi$, as $\Delta \subset \Delta_\phi$, and $\Delta_\phi$ is not crossed by $\phi(x;B)$.

	Now we check that this cell decomposition is uniformly definable. For each $\phi \in \Phi$, let $\mathcal{T}_\phi$ consist of $\Psi_\phi$ and $\{\theta_\psi: \psi \in \Psi_\phi\}$. Then $\mathcal{T}$ can be defined by the set of formulas $\Psi$ consisting of all conjunctions $\bigwedge_{\phi \in \Phi} \psi_\phi$ where $\psi_\phi \in \Psi_\phi$ for each $\phi$. For a given $\Delta = \bigcap_{\phi \in \Phi}\Delta_\phi$, we can let $\mathcal{I}(\Delta) = \bigcup_{\phi \in \Phi} \mathcal{I}(\Delta_\phi)$.
\end{proof}

Examples of distal structures include:
\begin{itemize}
	\item $o$-minimal structures
	\item Presburger arithmetic $(\Z,0,+,<)$
	\item The field of $p$-adics $\Q_p$ and other $P$-minimal fields.
	\item The linear reduct of $\Q_p$, in the language $\mathcal{L}_{\mathrm{aff}}$.
\end{itemize}
For justification of the first three of these, see \cite{cs}. The distality of these structures is established using the indiscernible sequence definition, which does not provide good bounds. In what follows, we will construct explicit distal cell decompositions for all of these examples.

\section{Dimension Induction}
In this section, we provide a bound on the size of distal cell decompositions for all dimensions, given a bound for distal cell decompositions for a fixed dimension in an arbitrary distal structure. This allows us to bound the size of a distal cell decomposition for any finite family of formulas in several kinds of distal structures, including any $o$-minimal structures. This approach is inspired by the partition construction in \cite{chaz91}, which can be interpreted as constructing distal cell decompositions in the context of $\R$ as an ordered field. (It also improves the bound in \cite[Proposition 1.9]{acgz}.)

\begin{thm}
\label{dimensioninduction}
	Let $\mathcal{M}$ be a structure in which all finite sets $\Phi(x;y)$ of formulas with $\abs{x}=1$ admit a distal cell decomposition with $k$ parameters (see Definition \ref{def_param_exp}), and for some $d_0\in \N$, all finite sets $\Phi(x;y)$ of formulas with $\abs{x}=d_0$ admit distal cell decompositions of exponent at most $r$. Then all finite sets $\Phi(x;y)$ of formulas with $\abs{x}=d\geq d_0$ admit distal cell decompositions of exponent $k(d-d_0)+r$.
\end{thm}
\begin{proof}
The case with $d=d_0$ follows directly from the assumptions, so we can proceed by induction. Assume the result for all finite sets of formulas with $\abs{x}=d-1 \geq d_0$. Then we will build a distall cell decomposition for a $\Phi(x;y)$ with $\abs{x}=d$. Where $x=(x_1,\dots,x_d)$, let $x'=(x_2,\dots,x_d)$. We start by fixing a distal cell decomposition $\mathcal{T}_1$ for the set of formulas $\Phi_1(x_1;x',y) := \{\phi(x_1;x',y):\phi(x;y) \in \Phi\}$. Let the cells of $\mathcal{T}_1$ be defined by $\Psi_1(x_1;x'_1,y_1,\dots, x'_k, y_k)$ and a formula $\theta_\psi(x',y;x'_1,y_1,\dots,x'_k,y_k)$ for each $\psi \in \Psi_1$. For this construction, we will only use $\mathcal{T}_1$ to define $\Phi_1$-types over sets of the form $\{a'\}\times B$. Because each element of that set has the same first coordinate, we will abbreviate the formula $\psi(x_1;x'_1,y_1,x'_2,y_2,\dots,x'_1,y_k)$ as $\psi(x_1;x',y_1,y_2,\dots,y_k)$, assuming all the variables $x'_i$ are equal. Similarly, we abbreviate $\theta_\psi(x',y;x'_1,y_1,\dots,x'_k,y_k)$ as $\theta_\psi(x',y;y_1,\dots,y_k)$, setting each $x'_i$ equal to $x'$. We will also want to repartition the variables, setting $\theta_\psi*(x';y_1,\dots,y_k,y) := \theta_\psi(x',y;y_1,\dots,y_k)$.

For each $\psi \in \Psi_1$, let $\Phi_\psi(x';y_1,\dots,y_k,y)$ be the set of formulas consisting of $\theta_\psi*$ and all formulas of the form $\A x_1, \psi(x_1;x',y_1,\dots,y_k) \to \square\phi(x_1,x';y)$ where $\phi \in \Phi$, and $\square$ is either $\neg$ or nothing.

Then let $\mathcal{T}_\psi$ be a distal cell decomposition for $\Phi_\psi$, consisting of $\Psi_\psi$ and a formula $\theta_{\psi'}$ for each $\psi' \in \Psi_\psi$. As before, we will assume some of the variables are equal, and write these formulas more succinctly, assuming that our set of parameters is of the form $\{(b_1,\dots,b_k)\}\times B$ for some $b_1,\dots,b_k \in M$ and finite $B \subseteq M^{\abs{y}}$. This allows us to write each $\psi' \in \Psi_\psi$ as $\psi'(x';y_1,\dots,y_k,y_1',\dots,y_m')$, and write $\theta_{\psi'}$ as $\theta_{\psi'}(y;y_1,\dots,y_k,y_1',\dots,y_m')$.

For each $\psi \in \Psi_1$ and $\psi' \in \Psi_\psi$, let $\psi\tensor \psi'(x;y_1,\dots,y_k,y'_1,\dots,y_m')$ be the formula $$\psi'(x';y_1,\dots,y_k,y'_1,\dots,y_m')\wedge \psi(x_1;x',y_1,\dots,y_k).$$
(Intuitively, this defines a sort of cylindrical cell in $M^{\abs{x}}$, where $x'$ is in a cell of one cell decomposition of $M^{\abs{x'}}$, and $x_1$ is in a cell of a cell decomposition of $M$, defined using $x'$ as a parameter.) Let $\Psi(x;y_1,\dots,y_k,y'_1,\dots,y_m') = \{\psi \tensor \psi' : \psi \in \Psi_1, \psi' \in \Psi_\psi\}$. We will use $\Psi$ to define a distal cell decomposition $\mathcal{T}$ for $\Phi(x;y)$.

To define $\mathcal{T}$, it suffices to define $\theta_{\psi \tensor \psi'}$ for each $\psi \in \Psi_1, \psi' \in \Psi_\psi$. Define
\begin{align*}
	\theta_{\psi \tensor \psi'}&(y;y_1,\dots,y_k,y'_1,\dots,y'_m) :=\\
&\theta_{\psi'}(y;y_1,\dots,y_k,y_1',\dots,y_m') \wedge (\E x', \psi'(x';y_1,\dots,y_k,y'_1,\dots,y'_m) \wedge \theta_\psi(x'; y_1,\dots,y_k,y)).
\end{align*}
This means that if $\Delta$ is the cell $\psi \tensor \psi'(M^d;b_1,\dots,b_k,b'_1,\dots,b_m')$, then
\begin{align*}
	\mathcal{I}(\Delta) &:=
\{b \in M^{\abs{y}} : \mathcal{M}\models \theta_{\psi \tensor \psi'}(b; b_1,\dots,b_k,b'_1,\dots,b_m')\}\\
 &= \{b \in M^{\abs{y}} : \exists (a_1,a') \in \Delta, \mathcal{M}\models \theta_\psi(a',b;b_1,\dots,b_k)\}.
\end{align*} Thus for all $a'$ in the projection of $\Delta$ onto $M^{d-1}$, the fiber $\{a_1 \in M: (a_1,a') \in \Delta\}$ is a cell of $\mathcal{T}_1(\{a'\}\times B)$ if and only if $B \cap \mathcal{I}(\Delta)=\emptyset$.

Now we show that this definition of $\mathcal{T}$ gives a valid distal cell decomposition for $\Phi(x;y)$. Fix a finite $B \subset M^{\abs{y}}$ and let $a \in M^d$ be given. Firstly, each element of $M^d$ is contained in a cell. If $a=(a_1,a')$ with $a_1 \in M, a' \in M^{d-1}$, then $a_1$ is in some cell of $\mathcal{T}_1(\{a'\}\times B)$, and that cell is defined by some $\psi(x_1;a',b_1,\dots,b_k)$, so for all $b \in B$, $\mathcal{M} \models \neg \theta_\psi*(a';b_1,\dots,b_k,b)$. Therefore $a'$ is in some cell of $\mathcal{T}_\psi(\{(b_1,\dots,b_k)\}\times B)$ on which $\mathcal{M} \models \neg \theta_\psi*(x';b_1,\dots,b_k,b)$. If that cell is defined by $\psi'(b_1,\dots,b_k,b_1',\dots,b_m')$, then we can now define a cell containing $a$ by $\psi \otimes \psi' (x;b_1,\dots,b_k,b_1',\dots,b_m')$.

Secondly, we show that each cell of $\mathcal{T}(B)$ is not crossed by $\Phi(x;B)$. Fix a cell $\Delta \in \mathcal{T}(B)$, and fix $\phi \in \Phi$, $b \in B$. We know that for each $a'$ in the projection of $\Delta$ onto $M^{d-1}$, the fiber $\{a_1 \in M: (a_1,a') \in \Delta\}$ is a cell of $\mathcal{T}_1(\{a'\} \times B)$, so that fiber is not crossed by $\phi(x;B)$. We also guaranteed that if $\Delta$ is defined by the formula $\psi \tensor \psi'(x,b_1,\dots,b_k,b_1',\dots,b_m')$, then the projection of $\Delta$ onto $M^{d-1}$ is a cell of $\mathcal{T}_{\psi}(B)$, so it is not crossed by the formulas $\A x_1, \psi(x_1; x',b_1,\dots,b_k) \to \phi(x_1,x';b)$ and $\A x_1, \psi(x_1; x',b_1,\dots,b_k) \to \neg \phi(x_1,x';b)$. If for some $(a_1,a')$ in $\Delta$, $\mathcal{M}\models \phi(a_1,a';b)$, then $\mathcal{M}\models \A x_1, \psi(x_1; x',b_1,\dots,b_k) \to \phi(x_1,x';b)$ for $x' = a'$, and thus for all $x'$ in the projection of $\Delta$, so $\mathcal{M}\models \psi(x;b)$ for all $x \in \Delta$.

Finally we can count the number of cells of $\mathcal{T}(B)$. For each $\psi \in \mathcal{T}_1$, and each $b_1,\dots,b_k$, there are, by induction, $\bigO(\abs{B}^{k((d-1)-d_0)+r})$ cells in $\mathcal{T}'(\{(b_1,\dots,b_k)\}\times B)$, each inducing a cell of $\mathcal{T}(B)$. Multiplying by the $\abs{B}^k$ possible tuples $(b_1,\dots,b_k) \in B^k$ and a finite number of formulas $\psi$, we get the desired bound $\bigO(\abs{B}^{k(d-d_0)+r})$.
\end{proof}

\section{Weakly $o$-Minimal Structures}
\label{omin}
In any structure $\mathcal{M}$, for any $n$, there is a formula $\phi(x;y)$ with $\abs{x} = n$ such that the the dual VC density of $\phi$ is $\abs{x}$, giving a lower bound on the distal density (see \cite[Section 1.4]{vcdensityi}). In this section, we construct an optimal distal cell decomposition for the case $\abs{x}=1$, and then use Theorem \ref{dimensioninduction} to construct distal cell decompositions for all $\Phi$, and bound their sizes. In the case where $\mathcal{M}$ is an $o$-minimal expansion of a group, we start instead with the optimal bound for $\abs{x} = 2$ from \cite{cs} and obtain a the bound on the size of the sign-invariant stratification in \cite{chaz91}, and improves 
the bounds on \cite[Theorem 4.0.9]{barone}.
\begin{thm}\label{ominthm}
	If $\Phi(x;y)$ is a finite family of formulas in a weakly $o$-minimal structure $\mathcal{M}$, then $\Phi$ admits a distal cell decomposition for $\Phi$ with exponent $2\abs{x}-1$.

	If $\mathcal{M}$ is an $o$-minimal expansion of a group and $\abs{x}\geq 2$, then the distal density is at most $2\abs{x}-2$.
\end{thm}
\begin{proof}
	In any weakly $o$-minimal structure, if $\Phi(x;y)$ has $\abs{x}=1$, then there exists a distal cell decomposition $\mathcal{T}$ with $\abs{\mathcal{T}(B)}=\bigO(\abs{B})$ with 2 parameters.

	Indeed, by weak $o$-minimality, for any $\varphi(x;y) \in \Phi$ with $\abs{x} = 1$, there is some number $N_\varphi$ such that the set $\varphi(M;b)$ is a union of at most $N_\varphi$ convex subsets for any $b \in M^{\abs{y}}$. Let $N:=\max_{\varphi \in \Phi}N_\varphi$. Then for each $\varphi(x;y) \in \Phi$, we can define formulas $\varphi^1(x;y),\dots,\varphi^N(x;y)$ by 
	\begin{align*}
		&\varphi^n(x;y) :=\\
		\E x_1,&x_2,\dots,x_{n-1},y_1,\dots,y_{n-1}, \varphi(x;y)\wedge (x_1 < y_1 < x_2 <\dots <y_{n-1}<x) \wedge \bigwedge_{i=1}^{n-1} (\varphi(x_i;y)\wedge \neg \varphi(y_i;y))
	\end{align*}
	and then
	$$\varphi(M;b)=\varphi^1(M;b)\cup \dots \cup \varphi^N(M;b)$$
	for all $b$, each $\varphi^i(M;b)$ is convex, and $\varphi^i(M;b) < \varphi^{i+1}(M;b)$ for each $i$, in the sense that for every $x_i \in \varphi^i(M;b)$ and $x_{i+1} \in \varphi^{i+1}(M;b)$, $x_i < x_{i+1}$.

	Then for each $\varphi \in \Phi$ we can also define 
	$$\varphi^i_\leq(x;y):=\E x_0(\varphi^i(x_0;y)\wedge x\leq x_0),$$
	$$\varphi^i_<(x;y):=\A x_0(\varphi^i(x_0;y)\to x< x_0).$$
		
	Note that each $\varphi^i_\square(M;b)$ for $\square\in \{<,\leq\}$ is closed downwards. Thus for any finite subset $B \subset M^{\abs{y}}$, the family of sets $\mathcal{F}(B)=\{\varphi^i_\square(M,b):b \in B,\varphi \in \Phi,1\leq i \leq N, \square \in \{<,\leq\}\}$ is linearly ordered under inclusion. Thus the atoms in the boolean algebra $\mathcal{B}$ generated by $\mathcal{F}(B)$ are of the form $X_1\setminus X_2$ where $X_1,X_2 \in \mathcal{F}(B)$ and $X_2$ is the unique maximal element of $\mathcal{F}(B)$ properly contained in $X_1$, or $\mathcal{M}\setminus X_1$ where $X_1$ is the unique maximal element of $\mathcal{F}(B)$. Thus only one atom of the boolean algebra can be of the form $X_1 \setminus X_2$ for each $X_1$, and thus the number of such atoms is at most $\abs{\mathcal{F}(B)}+1$, which is $\bigO(\abs{B})$.

	Now we construct $\mathcal{T}$. We let $\Psi$ consist of the formulas of the form
	$\psi(x;y_1,y_2):=\varphi^i_{\square_1}(x;y_1)\wedge \neg\varphi^j_{\square_2}(x;y_2)$ or $\psi(x;y):=\neg \varphi^j_{\square_1}(x;y)$ with $1\leq i \leq N, \square \in \{<,\leq\}$,
	and then for each potential cell $\Delta=\psi(M;b_1,b_2)$, let $\mathcal{I}(\Delta)$ just consist of all $b \in M^{\abs{y}}$ such that $\Delta$ is crossed by $\varphi_0(M;b)$ for some $\varphi_0 \in \Phi$. Then $\mathcal{T}(B)$ is exactly the set of atoms in the boolean algebra generated by $\mathcal{F}(B)$, so $\abs{\mathcal{T}(B)}=\bigO(\abs{B})$. Each cell is not crossed by any set in $\mathcal{F}(B)$, and thus not by any $\varphi(x;B)$, or $\Phi(x;B)$ itself, so this is a valid distal cell decomposition, where every cell is defined using at most 2 parameters from $B$.

	Thus we can use Theorem \ref{dimensioninduction}, setting $d_0=1$, $r=1$, and $k=2$, to find that any family of formulas $\Phi(x;y)$ has a distal cell decomposition of exponent at most $2(\abs{x}-1)+1=2\abs{x}-1$.

	If $\mathcal{M}$ is an $o$-minimal expansion of a group, we can instead set $d_0=2$, then we can set $r=2$, and by \cite[Theorem 4.1]{cgs}, for $\Phi(x;y)$ with $\abs{x}=2$, $\Phi$ admits a distal cell decomposition of exponent $2$. (In \cite{cgs}, this is only proven for the case where $\mathcal{M}$ is an expansion of a field, but the proof only uses it for definable choice, which $o$-minimal expansions of groups also have.) Then for $\abs{x}\geq 2$, $\Phi(x;y)$ admits a distal cell decomposition of exponent $2(\abs{x}-2)+2=2\abs{x}-2$.
\end{proof}

In the case of the ordered field $\R$, more is known. In that case, the distal cell decomposition produced in the above proof is the stratification in \cite{chaz91}.
An earlier version of that paper includes an improved bound for the case where
$\abs{x}=3$, showing that $\abs{\mathcal{T}(B)}=\bigO(\abs{B}^3\beta(\abs{B}))=\bigO(\abs{B}^{3+\varepsilon})$ for all $\varepsilon>0$, where $\beta$ is an extremely slowly growing function defined using the inverse of the Ackermann function.\cite{chaz89}
The argument uses Davenport-Schinzel sequences, purely combinatorial objects which lend themselves naturally to counting the complexity of cells defined by inequalities of a bounded family of functions.
The lengths of Davenport-Schinzel sequences can be bounded in terms of the inverse Ackermann function, giving rise to the $\beta(\abs{B})$ term.
For a general reference on such sequences, see \cite{davenport}.
These techniques are extended in \cite{4epsilon} to the case $\abs{x}=4$, where it is shown that $\abs{\mathcal{T}(B)}=\bigO(\abs{B}^{4+\varepsilon})$ for all $\varepsilon>0$.
These results imply that any finite set of formulas $\Phi(\abs {x};\abs {y})$ over $\R$ the ordered field has distal density 3 if $\abs{x}=3$, and $2\abs{x}-4$ if $\abs{x}\geq 4$.
It would be interesting to see if these bounds hold in any $o$-minimal structure, again using Davenport-Schinzel sequences.
It seems possible that every $\Phi(x;y)$ in an $o$-minimal structure has distal density $\abs{x}$, or admits a distal cell decomposition of exponent exactly $\abs{x}$, although new tools would be required to prove such claims.

\subsection{Locally Modular $o$-minimal Groups}
The trichotomy theorem for $o$-minimal structures classifies them locally into three cases: trivial, ordered vector space over an ordered division ring, and expansion of a real closed field \cite{trichotomy}. The $o$-minimal structures that are locally isomorphic to ordered vector spaces are known as the \emph{linear} structures, and can also be classified as those satisfying the \emph{CF property} \cite{linomin}. Any such structure must extend the structure of either an ordered abelian group or an interval in an ordered abelian group. We will show that with the added assumption of local modularity, all finite families of formulas in $o$-minimal expansions of groups admit optimal distal cell decompositions. This includes the special case of any ordered vector space over an ordered division ring.

\begin{thm}\label{ominvec}
	Let $\mathcal{M}$ be an $o$-minimal expansion of an 
	ordered group, with $\mathrm{Th}(\mathcal{M})$ locally modular. Let $\Phi(x;y)$ be a finite set of formulas in the language of $\mathcal{M}$. Then $\Phi$ admits a distal cell decomposition of exponent $\abs{x}$.
\end{thm}
To prove this theorem, we will need the following lemma:
\begin{lem}\label{conj_lem}
		Let $\mathcal{M}$ be an $\mathcal{L}$-structure.

		Let $\Phi(x;y)$ be a set of $\mathcal{L}$-formulas such that the negation of each $\varphi \in \Phi$ is a disjunction of other formulas in $\Phi$. Assume that for any nonempty finite $B \subset M^{\abs{y}}$ and $\varphi \in \Phi$, the conjunction $\bigwedge_{b \in B} \varphi(x;b)$ is equivalent to the formula $\varphi(x;b_0)$ for some $b_0 \in B$, or is not realizable. Then $\Phi$ admits a distal cell decomposition $\mathcal{T}$ such that for all finite $B$, the cells of $\mathcal{T}(B)$ are in bijection with the $\Phi$-types $S^{\Phi}(B)$. In particular, the distal density of $\Phi$ equals the dual VC density of $\Phi$.
	\end{lem}
	\begin{proof}
		Let $\Psi$ be the set of all formulas of the form $\psi(x;(y_\varphi)_{\varphi \in \Phi}) := \bigwedge_{\varphi \in \Phi'}\varphi(x;y_\varphi)$, where $\Phi' \subset \Phi$ is arbitrary.

		To define the distal cell decomposition $\mathcal{T}$, for each $\psi \in \Psi$, let $\theta_\psi(y;(y_\varphi)_{\varphi \in \Phi})$ denote $$\bigvee_{\varphi \in \Phi} \exists x, (\varphi(x_1;y) \wedge \psi(x;(y_\varphi)_{\varphi \in \Phi}) \wedge \neg \exists x_2, (\varphi(x_2;y) \wedge \psi(x;(y_\varphi)_{\varphi \in \Phi}).$$

		Then for a fixed finite $B \subset M^{\abs{y}}$, and fixed $b_\varphi: \varphi \in \Phi$ in $B$, let $\Delta$ be the cell $\psi(M;(b_\varphi)_{\varphi \in \Phi})$. Then for $b \in B$, we see that $b \in I(\Delta)$ if and only if the cell defined by $\psi(x;(b_\varphi)_{\varphi \in \Phi})$ is crossed by $\varphi(x;b)$ for some $\varphi \in \Phi$.

		We now claim that for any finite $B \subset M^{\abs{y}}$, the cells of $\mathcal{T}(B)$ correspond exactly to the $\Phi$-types $S^{\Phi}(B)$. As each cell $\Delta$ of $\mathcal{T}(B)$ is not crossed by $\Phi(B)$, its elements belong to a unique type of $S^{\Phi}(B)$. We claim that this type will be realized exactly by the elements of $\Delta$. This type is equivalent to a single formula, which will be of the form $\bigwedge_{\varphi \in \Phi}\left(\bigwedge_{b \in B} \square_{\varphi,b}\varphi(x;b)\right)$, where each $\square_{\varphi,b}$ is either $\neg$ or nothing. For each $\varphi,b$ such that $\square_{\varphi,b}$ is $\neg$, we may simply drop $\neg\varphi(x;b)$ from the conjunction, because $\neg \varphi(x;b)$ is equivalent to the disjunction $\bigvee_{\varphi \in \Phi_\varphi} \varphi(x;b)$ for some subset $\Phi_\varphi \subseteq \Phi$, and as the type is realizable, $\varphi_{i_0}(x;b)$ rather than its negation must already appear in the conjunction for some $i_0$, and we can replace $\varphi_{i_0}(x;b) \wedge \bigvee_{\varphi \in \Phi_\varphi} \varphi(x;b)$ with simply $\varphi_{i_0}(x;b)$. In this way, inductively, we can continue to remove all of the negated formulas in the conjunction, until we are left with $\bigwedge_{\varphi \in \Phi'}\left(\bigwedge_{b \in B_\varphi} \varphi(x;b)\right)$ where $\Phi'\subseteq \Phi$, and each $B_\varphi \subseteq B$ is nonempty. By our other assumption, as this formula is realizable, it is equivalent to $\bigwedge_{\varphi \in \Phi'}\square_{\varphi, b_\varphi} \varphi(x;b_\varphi)$ where each $b_\varphi \in B$, which in turn is a defining formula for a cell of $\mathcal{T}(B)$, which must be $\Delta$.
\end{proof}

\begin{proof}[Proof of Theorem \ref{ominvec}]
	By $o$-minimality, we can assume the group is abelian.
	Let $\mathcal{L}_{\mathcal{M}}$ be the language of $\mathcal{M}$.
 	Corollary 6.3 of \cite{linomin} shows that $\mathcal{M}$ admits quantifier elimination in the language $\mathcal{L}'$, consisting of $+, <$, the set of algebraic points (that is, $\mathrm{acl}(\emptyset)$) as constants, and a unary function symbol for each $0$-definable partial endomorphism of $\mathcal{M}$. Recall that a partial endomorphism is defined as a function of type either $f : M \to M$ or $f : (-c, c) \to M$ for some $c \in M$, such that if $a,b, a + b$ are all in the domain, then $f(a+b) = f(a)+f(b)$. The unary symbols representing the partial endomorphisms are assigned the value 0 outside the domain. If $f$ has domain $(-c,c)$, then $c \in \mathrm{acl}(\emptyset)$. Note that by $o$-minimality, $\mathrm{acl} = \mathrm{dcl}$, so each of the constants in this language is in $\mathrm{dcl}(\emptyset)$, so each symbol of this language is $\emptyset$-definable in the original structure $(\mathcal{M},\mathcal{L}_{\mathcal{M}})$.

 	Each formula in $\Phi$ is equivalent modulo $\mathrm{Th}(\mathcal{M})$ to some formula in $\mathcal{L}'$, so we replace $\Phi$ with $\Phi_{\mathcal{L}'}$, a pointwise equivalent finite set of $\mathcal{L}$-formulas. It suffices to find a distal cell decomposition of exponent $|x|$ for $\Phi_{\mathcal{L}'}$. As the interpretation of every symbol of $\mathcal{L}'$ is $\emptyset$-definable in $\mathcal{L}_{\mathcal{M}}$, we can replace each formula of this distal cell decomposition with an equivalent $\mathcal{L}_{\mathcal{M}}$-formula without parameters.

 	By quantifier elimination in $\mathcal{L}'$, we can find a finite set of atomic $\mathcal{L}'$-formulas $\Phi_A$ such that each formula in $\Phi_{\mathcal{L}'}$ is equivalent to a boolean combination of formulas in $\Phi_A$ modulo $\mathrm{Th}(\mathcal{M})$. Lemma \ref{bool_comb} tells us that a distal cell decomposition for $\Phi_A$ is a distal cell decomposition for $\Phi_{\mathcal{L}'}$, so it suffices to prove the desired result for $\Phi_A$. We then will find another finite set of $\mathcal{L}'$-formulas, $\Phi'$, such that each atomic formula in $\Phi_A$ is a boolean combination of formulas in $\Phi'$, and $\Phi'$ satisfies the conditions of the following lemma, providing us with a distal cell decomposition that we can show has the desired exponent. It suffices to find $\Phi'$ satisfying the requirements of Lemma \ref{conj_lem} such that any atomic formula in $\Phi_A$ is a boolean combination of formulas from $\Phi'$, and to show that for any finite $\Phi$ and $B$, $\abs{S^{\Phi}(B)} \leq \bigO(\abs{B}^{\abs{x}})$.

	We will select $\Phi'$ to contain only atomic $\mathcal{L}'$-formulas of the form $f(x) + g(y) + c \square 0$, where $f,g$ are group endomorphisms, $c$ is a term built only out of functions and constants, and $\square \in \{<,=,>\}$. If $\varphi(x;y)$ is of the form $f(x) + g(y) + c = 0$, then for a given $B$, $\bigwedge_{b \in B}\varphi(x;b)$ is either equivalent to $\varphi(x;b)$ for all $b \in B$ or not realizable. If $\varphi$ is an inequality, then $\bigwedge_{b \in B} \varphi(x;b)$ is equivalent to $\varphi(x;b_0)$ for some $b_0$ minimizing or maximizing $g(b)$. Also, for all $\varphi \in \Phi'$, $\neg \varphi(x;y)$ is a disjunction of other formulas in $\Phi'$, because $\neg f(x) + g(y) + c = 0$ is equivalent to $f(x) + g(y) + c < 0 \vee f(x) + g(y) + c > 0$, $\neg f(x) + g(y) + c < 0$ is equivalent to $f(x) + g(y) + c = 0 \vee f(x) + g(y) + c > 0$, and $\neg f(x) + g(y) + c > 0$ is equivalent to $f(x) + g(y) + c = 0 \vee f(x) + g(y) + c < 0$.

	Now we show that every atomic $\mathcal{L}'$-formula, and thus every formula in $\Phi_A$, can be expressed as a boolean combination of atomic formulas of the form $f(x) + g(y) + c \square 0$ with $f$ and $g$ total (multivariate) definable endomorphisms. Any atomic formula is of the form $f(x;y) \square g(x;y)$, and by subtraction is equivalent to $(f-g)(x;y) \square 0$. Thus it suffices to show that for any $\mathcal{L}'$-term $t(x;y)$ and $\square \in \{<,=,>\}$, the atomic formula $t(x;y) \square 0$ is equivalent to a boolean combination of formulas of the form $f(x) + g(y) + c \square' 0$ with $f$ and $g$ total endomorphisms and $\square' \in \{<,=,>\}$.

	We prove this by induction on the number of partial endomorphism symbols in $t(x;y)$ that do not represent total endomorphisms. If that number is 0, then every symbol in the term $t(x;y)$ is a variable, a constant, or represents a total endomorphism. Thus $t(x;y)$ is a composition of affine functions, and is thus itself an affine function, which can be represented as $f(x) + g(y) + c$. Thus $t(x;y) \square 0$ is equivalent to $f(x) + g(y) + c \square 0$. Now let $t(x;y)$ contain $n+1$ partial endomorphism symbols. Let one of them be $f$, so that $t(x;y) = t_1(f(t_2(x;y)),x,y)$ for some terms $t_1, t_2$. By \cite[Lemma 4.3]{linomin} and local modularity, $\mathcal{L}'$ contains a partial endomorphism symbol $g$ representing a total function such that $f(x) = g(x)$ on the interval $(-c,c)$, with $f(x) = 0$ outside of that interval. Thus $t(x;y) \square 0$ is equivalent to
	\begin{align*}
	    \left(-c < t_2(x;y) \wedge t_2(x;y)<c \wedge t_1(g(t_2(x;y)),x,y) \square 0\right)\\
	    \vee \left(\neg(-c < t_2(x;y) \wedge t_2(x;y)<c) \wedge t_1(0,x,y) \square 0\right).
	\end{align*}
	This is equivalent to a boolean combination of $t_2(x;y) + c > 0, t_2(x;y) - c < 0, t_1(g(t_2(x;y)),x,y)) \square 0$, and $t_1(0,x,y)\square 0$, each of which has at most $n$ non-total partial endomorphisms, and thus by induction, is a boolean combination of formulas of the desired form.

	Now we wish to verify that $\abs{S^{\Phi}(B)} \leq \bigO(\abs{B}^{\abs{x}})$. Theorem 6.1 of \cite{vcdensityi} says that the dual VC density of $\Phi$ will be at most $\abs{x}$, which is only enough to show that $\Phi$ has distal density $\abs{x}$. However, the proof shows that $\abs{S^{\Phi}(B)} \leq \bigO(\abs{B}^{\abs{x}})$. Tracing the logic of that paper, Theorem 6.1 guarantees that a weakly $o$-minimal theory has the VC1 property, which by Corollary 5.9 implies that $\Phi$ has uniform definition of $\Phi(x;B)$ types over finite sets with $\abs{x}$
	parameters, which implies that $\abs{S^\Phi(B)}\leq \bigO(\abs{B}^{\abs{x}})$ (as noted at the end of Section 5.1).

\end{proof}

\section{Presburger Arithmetic}\label{pres}
Presburger arithmetic is the theory of $\Z$ as an ordered group. As mentioned in Example 2.9 of \cite{cs}, the ordered group $\Z$ admits quantifier elimination in the language 
$\mathcal{L}_{\mathrm{Pres}} = \{0, 1, +, -, <, \{k\mid\}_{k \in \N}\}$, where for each $k \in \N$ and $x \in \Z$, $\Z \models k \mid x$ when $x$ is divisible by $k$, so we will work in this language. As this structure is quasi-$o$-minimal, it is distal, and we will construct an explicit distal cell decomposition with optimal bounds, similar to the distal cell decomposition for $o$-minimal expansions of locally modular ordered groups in Theorem \ref{ominvec}.

\begin{thm}\label{presthm}
	Let $G$ be an ordered abelian group with quantifier elimination in $\mathcal{L}_{\mathrm{Pres}}$. Let $\Phi(x;y)$ be a finite set of formulas in this language. Then $\Phi$ has distal density at most $\abs{x}$.
\end{thm}
\begin{proof}
	Throughout this proof, we will identify $\Z$ with the subgroup of $G$ generated by the constant 1.

	As $G$ has quantifier elimination in this language, every $\varphi(x;y) \in \Phi$ is equivalent to a boolean combination of atomic formulas. 
	We will group the atomic formulas into two categories. The first is those of the form $f(x) \square g(y) + c$, where $\square \in \{<,=,>\}$, $(f,g)$ belongs to a finite set $F$ of pairs of $\Z$-linear functions of the form $\sum_{i = 1}^{\abs{x}} a_i x_i$ with $a_i \in \Z$, and $c$ belongs to a finite set $C \subseteq \Z$. The second is atomic formulas of the form $k \mid (f(x) + g(y) + c)$ for $k \in \N$, $(f,g) \in F$, and $c \in C$. Furthermore, we may assume that only one symbol of the form $k \mid$ is used. If $K$ is the least common multiple of the finite collection of $k$ such that $k \mid$ appears in one of these atomic formulas, then each $k \mid (f(x) + g(y)+ c)$ can be replaced with $K \mid (d \cdot f(x) + d \cdot g(y) + d \cdot c)$, where $d \cdot \sum_{i = 1}^{\abs{x}}a_i x_i = \sum_{i = 1}^{\abs{x}}(d \cdot a_i) x_i$ and $dk = K$. Note that all of these functions and constants are $\emptyset$-definable.

	Then, by Lemma \ref{bool_comb}, we may replace $\Phi(x;y)$ with the union of the following two sets of atomic formulas for appropriate choices of $F$ and $C$:
	\begin{itemize}
		\item Fix $C$ to be a finite subset of $\Z$, $F$ a finite subset of pairs of $\Z$-linear functions of the form $\sum_{i = 1}^{\abs{x}} a_i x_i$, and $K \in \N$.
		\item Let $\Phi_0$ be the set of all $f(x)\square g(y)+ c$  with $(f,g)\in F,c \in C, \square \in \{<,=,>\}$.
		\item Let $\Phi_1$ be the set of all $K \mid (f(x)+g(y)+c)$ with $(f,g) \in F, c \in \{0,\dots,K-1\}$.
		\item Let $\Phi = \Phi_0 \cup \Phi_1$.
	\end{itemize}

	It is straightforward to see that the negation of any formula from $\Phi_0$ is equivalent to the disjunction of two formulas from $\Phi_0$, and a negation of any formula $K | (f(x)+g(y)+c)$ from $\Phi_1$ is equivalent to $\bigvee_{0 \leq c' <K, c' \neq c} K|(f(x)+g(y)+c)$, a disjunction of formulas from $\Phi_1$.

	To apply Lemma \ref{conj_lem}, it suffices to show that for any $\varphi \in \Phi$ and nonempty finite $B \subset M^{\abs{y}}$, $\bigwedge_{b \in B} \varphi(x;b)$ is equivalent to $\varphi(x;b_0)$ for some $b_0 \in B$ or is not realizable. This holds for $\varphi \in \Phi_0$ for reasons discussed in the proof of \ref{ominvec}. For $\varphi \in \Phi_1$, we see that if there exist $b_1, b_2$ such that $g(b_1) \not\equiv g(b_2) \pmod{K}$, then $\varphi(x;b_1) \wedge \varphi(x;b_2)$ implies $K|(f(x)+g(b_1)+c) \wedge K|(f(x)+g(b_2)+c)$ so $K|\left(g(b_1)-g(b_2)\right)$, a contradiction. Thus this conjunction is not realizable. Otherwise, for any $b_0 \in B$, and any other $b \in B$, $g(b) \equiv g(b_0)\pmod{K}$, so $\bigwedge_{b \in B}\varphi(x;b)$ is equivalent to $\varphi(x;b_0)$.

	Now Lemma \ref{conj_lem} gives us a distal cell decomposition $\mathcal{T}$ for $\Phi$, such that for all $B$, $\abs{T(B)} = \abs{S^{\Phi}(B)}$. The theory of $\Z$ in $\mathcal{L}_{\mathrm{Pres}}$ is quasi-$o$-minimal by \cite[Example 2]{quasio}, and the same argument will hold for $G$, because $G$ has quantifier elimination in the same language. The same VC density results apply to quasi-$o$-minimal theories as to $o$-minimal theories (see \cite[Theorem 6.4]{vcdensityi}), so $\abs{S^{\Phi}(B)}\leq \bigO(\abs{B}^{\abs{x}})$.
\end{proof}

\section{$\Q_p$, the linear reduct}
\label{qpaff}
Now we turn our attention to the linear reduct of $\Q_p$, viewed as a structure $\mathcal{M}$ in the language $\mathcal{L}_\mathrm{aff}=\{0,+,-,\{c \cdot \}_{c \in \Q_p},\mid,\{Q_{m,n}\}_{m,n\in\N \setminus \{0\}}\}$, where $c \cdot$ is a unary function symbol which acts as scalar multiplication by $c$, $x\mid y$ stands for $v(x)\leq v(y)$, and $\mathcal{M}\models Q_{m,n}(a)$ if and only if $a \in \bigcup_{k \in \Z}p^{km}(1+p^n\Z_p)$. For each $m,n$, the set $Q_{m,n}(M) \setminus \{0\}$ is a subgroup of the multiplicative group of $\Q_p$ with finite index. Leenknegt \cite{lee12,lee} introduced this structure (referring to the language as $\mathcal{L}_\mathrm{aff}^{\Q_p}$), proved that it is a reduct of Macintyre's standard structure on $\Q_p$, and proved cell decomposition results for it which imply quantifier elimination.

Bobkov \cite{bobkov} shows that every finite set $\Phi(x;y)$ of formulas has dual VC density $\leq\abs{x}$, and this section is devoted to strengthening this by proving the same optimal bound for the distal density:
\begin{thm}\label{qpaffthm}
 	For any finite set $\Phi(x;y)$ of $\mathcal{L}_\mathrm{aff}$-formulas in $\Q_p$, there is a distal cell decomposition $\mathcal{T}$ with $\abs{\mathcal{T}(B)}=\bigO(\abs{B}^{\abs{x}})$, so $\Phi$ has distal density $\leq \abs{x}$.
\end{thm}

It is worth noting that Bobkov used a slightly different version of this language, which included the constant 1, therefore making all definable sets $\emptyset$-definable. Because our distal cell decomposition must be definable without parameters, we will use slightly stronger versions of Leenknegt and Bobkov's basic lemmas, to avoid parameters. The first such result is a cell-decomposition result, proven in \cite{lee12}, but stated most conveniently as \cite[Theorem 4.1.5]{bobkov}. To state it, we need to define what a cell is in that context:

\begin{defn}
	A \emph{$0$-cell} is the singleton $\Q^0_p$. A \emph{$(k + 1)$-cell} is a subset of $\Q^{k+1}_p$ of the following form:
$$\{(x, t) \in D \times \Q_p \, | \, v(a_1 (x)) \, \square_1 \, v(t - c(x)) \, \square_2 \, v(a_2(x)), \, t - c(x) \in \lambda Q_{m,n}\},$$ where $D$ is a $k$-cell, $a_1, a_2, c$ are polynomials of degree $\leq 1$, called the \emph{defining polynomials}, each of $\square_1, \square_2$ is either $<$ or no condition, $m,n \in \N$, and $\lambda \in \Q_p$.
\end{defn}

\begin{fact}[{\cite{lee12}, see also \cite[Theorem 4.1.5]{bobkov}}]
	Any definable subset of $\Q^k_p$ (in the language $\mathcal{L}_{\mathrm{aff}}$) decomposes into a finite disjoint union of $k$-cells.
\end{fact}

Now we modify these definitions and results to work in an $\emptyset$-definable context:

\begin{defn}
	A \emph{$0$-cell over $\emptyset$} is just a $0$-cell.
	A \emph{$(k+1)$-cell over $\emptyset$} is a $(k+1)$-cell $\{(x, t) \in D \times \Q_p | v(a_1 (x)) \square_1 v(t - c(x)) \square_2 v(a_2(x)), t - c(x) \in \lambda Q_{m,n}\}$
	where $D$ is a $k$-cell over $\emptyset$ and the defining polynomials have constant coefficient 0.
\end{defn}

We can now state a $\emptyset$-definable version of the cell decomposition result:
\begin{lem}\label{empty_cell}
	Any $\emptyset$-definable subset of $\Q^k_p$ (in the language $\mathcal{L}_{\mathrm{aff}}$) decomposes into a finite disjoint union of $k$-cells over $\emptyset$.
\end{lem}
\begin{proof}
	We trace the proof of the original cell decomposition result in \cite{lee12}. Lemmas 2.3 and 2.7 establish that finite unions of cells (in the case of finite residue field, equivalent to the ``semi-additive sets'' of Definition 2.6) are closed under intersections and projections respectively, and Lemma 2.5 (using Lemma 2.4) shows that all quantifier-free definable sets are semi-additive. It suffices to modify each of these four lemmas slightly. In all four lemmas, we modify the assumptions to require that all linear polynomials in the assumptions have constant term 0. In each construction, the polynomials in the results are linear combinations of the polynomials in the assumptions, and thus will also have constant term 0, allowing us to state the results in terms of $k$-cells \emph{over $\emptyset$}.
\end{proof}

This tells us that no nonzero constants are definable:
\begin{lem}
	In the structure $\mathcal{M}$ consisting of $\Q_p$ in the language $\mathcal{L}_\mathrm{aff}$, $\mathrm{dcl}(\emptyset) = \{0\}$.
\end{lem}
\begin{proof}
	If $a \in \mathrm{dcl}(\emptyset)$, then $\{a\}$ is $\emptyset$-definable, so it can be decomposed into $1$-cells over $\emptyset$. There can only be one cell in the decomposition, $\{a\}$. All of its defining polynomials take in variables from the unique $0$-cell, and thus consist only of their constant coefficient, which is 0. Thus the cell must be of the form
	$\{a\} = \{t \in D \times \Q_p | v(0) \square_1 v(t - 0) \square_2 v(0), t - 0 \in \lambda Q_{m,n}\}$. The condition $v(0) \square_1 v(t) \square_2 v(0)$ will define one of the following sets: $\emptyset, \{0\}, \Q_p \setminus \{0\}, \Q_p$, and the condition $t \in \lambda Q_{m,n}$ defines $\{0\}$ when $\lambda = 0$, and otherwise, $\lambda Q_{m,n} \subseteq \Q_p \setminus \{0\}$. Thus the whole cell is either $\{0\}$ or $\lambda Q_{m,n}$ which is infinite, so if it is a singleton $\{a\}$, we must have $a = 0$.
\end{proof}

We now check that our cell decomposition for $\emptyset$-definable sets yields $\emptyset$-definable cells:
\begin{lem}
	Any $k$-cell over $\emptyset$ is $\emptyset$-definable.
\end{lem}
\begin{proof}
	We prove this by induction on $k$. The $k = 0$ case is trivial. The $(k + 1)$-cell $\{(x, t) \in D \times \Q_p | v(a_1 (x)) \square_1 v(t - c(x)) \square_2 v(a_2(x)), t - c(x) \in \lambda Q_{m,n}\}$ is $\emptyset$-definable if $D$ is, $v(a_1 (x)) \square_1 v(t - c(x)) \square_2 v(a_2(x))$ is, and $t - c(x) \in \lambda Q_{m,n}$ is. We have that $D$ is by the induction hypothesis. For the next condition, it suffices to observe that the defining polynomials are $\emptyset$-definable functions if and only if they have constant coefficient 0, because scalar multiplication is $\emptyset$-definable, but no constant other than 0 is. For the final condition, we see that if $\lambda = 0$, then $t - c(x) \in \lambda Q_{m,n}$ is equivalent to $t - c(x) = 0$, which is $\emptyset$-definable, and if $\lambda \ne 0$, then $t - c(x) \in \lambda Q_{m,n}$ is equivalent to $\lambda^{-1}\cdot (t - c(x)) \in Q_{m,n}$, which is $\emptyset$-definable.
\end{proof}

We now want to generalize the following quantifier-elimination result to the $\emptyset$-definable case:
\begin{lem*}[{\cite[Theorem 4.2.1]{bobkov}}]
	Any $\mathcal{L}_{\mathrm{aff}}$-formula (with parameters) $\phi(x;y)$ where $x$ and $y$ are finite tuples of variables is equivalent in the $\mathcal{L}_{\mathrm{aff}}$-structure $\Q_p$ to a boolean combination of formulas from a collection 
	$$\Phi_\phi=\{v(p_i(x)-c_i(y))<v(p_j(x)-c_j(y))\}_{i,j \in I}\cup \{p_i(x)-c_i(y) \in \lambda Q_{m,n}\}_{i\in I,\lambda \in \Lambda}$$
	where $I=\{1,\dots,\abs{I}\}$ is a finite index set, each $p_i$ is a degree $\leq 1$ polynomial with constant term 0, each $c_i$ is a degree $\leq 1$ polynomial, and $\Lambda$ is a finite set of coset representatives of $Q_{m,n}$ for some $m,n \in \N$.
\end{lem*}

Bobkov derives this result from the cell decomposition. If we apply the same logic to the $\emptyset$-definable cell decomposition from Lemma \ref{empty_cell}, then all of the polynomials involved have constant term 0, and thus all formulas involved are $\emptyset$-definable:
\begin{lem}\label{empty_qe}
	Any $\mathcal{L}_{\mathrm{aff}}$-formula $\phi(x;y)$ where $x$ and $y$ are finite tuples of variables is equivalent in the $\mathcal{L}_{\mathrm{aff}}$-structure $\Q_p$ to a boolean combination of formulas from a collection 
	$$\Phi_\phi=\{v(p_i(x)-c_i(y))<v(p_j(x)-c_j(y))\}_{i,j \in I}\cup \{p_i(x)-c_i(y) \in \lambda Q_{m,n}\}_{i\in I,\lambda \in \Lambda}$$
	where $I=\{1,\dots,\abs{I}\}$ is a finite index set, each $p_i$ and each $c_i$ is a degree $\leq 1$ polynomial with constant term 0 and $\Lambda$ is a finite set of coset representatives of $Q_{m,n}$ for some $m,n \in \N$.
\end{lem}

As a corollary of this lemma and Lemma \ref{bool_comb}, we see that we can replace $\Phi$ with the set $\bigcup_\phi \Phi_\phi$, and thus assume that $\Phi$ takes the form
$$\{v(p_i(x)-c_i(y))<v(p_j(x)-c_j(y))\}_{i,j \in I}\cup \{p_i(x)-c_i(y) \in \lambda Q_{m,n}\}_{i\in I,\lambda \in \Lambda}.$$
for some fixed $m,n \in \N$.

We now recall some terminology from Bobkov \cite{bobkov}.

\begin{defn} [\cite{bobkov}, Def. 4.2.3]\label{subinterval}
	For the rest of this section, we fix $B \subset M^{\abs{y}}$, and let $T=\{c_i(b):i \in I,b \in B\}$.

\begin{itemize}
	\item For $c \in \Q_p$ and $r \in \Z$, we define $B_r(c):=\{x:v(x-c)>r\}$ and refer to it as the \emph{open} ball of radius $r$ around $c$.
	\item Let the \emph{subintervals} over a parameter set $B$ be the atoms in the Boolean algebra generated by the set of balls
$$\mathcal{B}:=\{B_{v(c_i(b_1)-c_j(b_2))}(c_i(b_1)):i,j \in I,b_1,b_2 \in B\}\cup\{B_{v(c_j(b)-c_k(b))}(c_i(b)):i,j,k \in I,b \in B\}$$
	\item Each subinterval can be expressed as $I(t,\alpha_L,\alpha_U)$ where
$$I(t,\alpha_L,\alpha_U)=B_{\alpha_L}(t)\setminus \bigcup_{t' \in T\cap B_{\alpha_U-1}(t)} B_{\alpha_U}(t'),$$
for some $t=c_i(b_0)$ with $i \in I, b_0 \in B$, and $\alpha_L=\alpha_1(b_0,b_1),\alpha_U=\alpha_2(b_0,b_2)$, with $\alpha_1,\alpha_2$ chosen from a finite set $A$ of $\emptyset$-definable functions $\Q_p^2 \to \Gamma$, including two functions defined, by abuse of notation, as $\pm \infty$.
	\item The subinterval $I(t,\alpha_L,\alpha_U)$ is said to be \emph{centered} at $t$.
\end{itemize}
\end{defn}

By this definition, it is not clear that $I(t,\alpha_L,\alpha_U)$ should be uniformly definable from parameters in $B$, as the set $T \cap B_{\alpha_U-1}(t)$ could depend on all of $B$. However, we can eliminate most of the balls from that definition. The ball $B_{\alpha_U-1}(t)$ can be split into $p$ balls of the form $B_{\alpha_U}(t')$ for some $t' \in \Q_p$, call them $B_1,\dots,B_p$. Let $T'$ be a subset of $T \cap B_{\alpha_U-1}(t)$ such that for each $B_i$, if $T \cap B_i\neq \emptyset$, then $T'$ contains only a single representative $t_i$ from $B_i$. Then 
$$\bigcup_{t' \in T \cap B_{\alpha_U-1}(t)}B_{\alpha_U}(t')=\bigcup_{t' \in T'}B_{\alpha_U}(t'),$$
because each $t' \in T \cap B_{\alpha_U-1}(t)$ belongs to some $B_i$, so $B_{\alpha_U}(t')=B_i=B_{\alpha_U}(t_i)$. We may assume $\abs{T'}$ to be at most $p-1$, because if all $p$ balls were removed, we could instead define this set as $I(t,\alpha_L,\alpha_U-1)$.
Thus each subinterval can be defined as $I(t,\alpha_L,\alpha_U)=\psi_{\mathrm{sub}}(t,\alpha_L,\alpha_U,\bar b)$, where $\psi_{\mathrm{sub}}$ is one of a finite collection $\Psi_{\mathrm{sub}}$ of formulas, and $\bar b$ is a tuple of at most $p-1$ elements of $B$.

\begin{defn} [\cite{bobkov}, Def. 4.2.5] \label{Tval}
	For $a \in \Q_p$, define $T\mathrm{-val}(a) := v(a-t)$, where $a$ belongs to a subinterval centered at $t$. By Lemma 4.2.6, \cite{bobkov}, this is well-defined, as $v(a-t)$ is the same for all valid choices of $t$.
\end{defn}

\begin{defn}[\cite{bobkov}, Def. 4.2.8]
	Given a subinterval $I(t, \alpha_L, \alpha_U)$, two points $a_1,a_2$ in that subinterval are defined to have the same \emph{subinterval type} if one of the following conditions is satisfied:
\begin{itemize}
	\item $\alpha_L+n\leq T\mathrm{-val}(a_i)\leq \alpha_U-n$ for $i=1,2$ and $(a_1-t)(a_2-t)^{-1}\in Q_{m,n},$
	\item $\neg(\alpha_L+n\leq T\mathrm{-val}(a_i)\leq \alpha_U-n)$ for $i=1,2$ and $T\mathrm{-val}(a_1)=T\mathrm{-val}(a_2)\leq v(a_1-a_2)-n.$
\end{itemize}
\end{defn}

We show that the set of points of each subinterval type is definable over $t,\alpha_L,\alpha_U$. The subinterval types of the first kind are definable by $$\psi_{\mathrm{tp}}^\lambda(x;t,\alpha_L,\alpha_U):=(\alpha_L+n\leq v(x-t)\leq \alpha_U-n)\wedge (x-t) \in \lambda Q_{m,n}$$
where $\lambda \in \Lambda$. The subinterval types of the second kind are definable by one of
$$\psi_{\mathrm{tp}}^{L,i,q}(x;t,\alpha_L,\alpha_U):=(v(x-t)=\alpha_L + i)\wedge (\alpha_L + i+n\leq v(x-(p^{\alpha_L + i} q + t)))$$
or
$$\psi_{\mathrm{tp}}^{U,i,q}(x;t,\alpha_L,\alpha_U):=(v(x-t)=\alpha_U - i)\wedge (\alpha_U - i+n\leq v(x-(p^{\alpha_U - i} q + t))),$$
where $0\leq i < n$, and $q$ ranges over a set $Q$ of representatives of the balls of radius $n$ contained in $B_0(0)\setminus B_{1}(0)$. If we let $\alpha$ be $\alpha_L + i$ or $\alpha_U - i$, this makes $p^\alpha q + t$ range over a finite set of representatives of the balls of radius $\alpha+n$ contained in the set $B_\alpha(t)\setminus B_{\alpha+1}(t)$ of points $a$ with $v(a-t)=\alpha$. Let $\Psi_{\mathrm{tp}}$ be the set of all these formulas: $\{\psi_{\mathrm{tp}}^\lambda: \lambda \in \Lambda\}\cup \{\psi_{\mathrm{tp}}^{L,i,q}:0 \leq i < n,q \in Q\}\cup \{\psi_{\mathrm{tp}}^{U,i,q}:0 \leq i < n,q \in Q\}$.

\subsection{Defining the Distal Cell Decomposition}
We start by defining $\Psi(x;(y_{0,i}: i \in I), (y_{1,i}: i \in I), (y_{2,i}: i \in I))$ to be the set of all formulas $\psi(x;(y_{0,i}: i \in I), (y_{1,i}: i \in I), (y_{2,i}: i \in I))$ of the form
\begin{align*}
	\left(\bigwedge_{i \in I}\left[\psi_{\mathrm{sub}}^i(p_i(x),t_i,\alpha_{L,i},\alpha_{U,i},\bar y_i)\wedge \psi_{\mathrm{tp}}^i(p_i(x),t_i,\alpha_{L,i},\alpha_{U,i})\right]\right) \wedge \psi_{\sigma}(x;t_1,\dots,t_{\abs{I}})
\end{align*}
where $\psi_{\mathrm{sub}}^i\in \Psi_{\mathrm{sub}}$, $\psi_{\mathrm{tp}}^i \in \Psi_{\mathrm{tp}}$, $\psi_{\sigma}(x,t_1,\dots,t_{\abs{I}})$ is, for some permutation $\sigma$ of $I$,
$$v(p_{\sigma(1)}(x)-t_{\sigma(1)})>\dots >v(p_{\sigma(\abs{I})}(x)-t_{\sigma(\abs{I})}),$$
and we define $t_i,\alpha_{L,i},\alpha_{U,i}$ so that $t_i = c_j(y_{0,i})$ for some $j \in I$, $\alpha_{L,i} = \alpha_1(y_{0,i}, y_{1,i})$, and $\alpha_{L,i} = \alpha_2(y_{0,i}, y_{2,i})$ for some $\alpha_1, \alpha_2 \in A$.

For each potential cell $\Delta$, we will define $\mathcal{I}(\Delta)$ so that $\Delta$ will be included in $\mathcal{T}(B)$ exactly when each set $\psi_{\mathrm{sub}}^i(M,t_i,\alpha_{L,i},\alpha_{U,i},\bar b_i)$ is actually a subinterval. Then each cell of $\mathcal{T}(B)$ will consist of all elements $a \in M^{\abs{x}}$ such that for all $i$, $p_i(a)$ belongs to a particular subinterval and has a particular subinterval type, and the set $\{T-\mathrm{val}(p_i(a)): i \in I\}$ has a particular ordering. These cells are not crossed by $\Phi(x;B)$, as a consequence of the following lemma:
\begin{lem}[{\cite[Lemma 4.2.12]{bobkov}}]
	Suppose $d, d' \in \Q_p$ satisfy the following three conditions:
\begin{itemize}
	\item For all $i \in I$, $p_i (d)$ and $p_i (d' )$ are in the same subinterval.
	\item For all $i \in I$, $p_i (d)$ and $p_i (d' )$ have the same subinterval type.
	\item For all $i, j \in I$, $T -\mathrm{val}(p_i (d)) > T -\mathrm{val}(p_j (d))$ iff $T -\mathrm{val}(p_i (d' )) > T -\mathrm{val}(p_j (d' ))$.

Then $d, d'$ have the same $\Phi$-type over $B$.
\end{itemize}
\end{lem}

Now we check that we can actually define $\mathcal{I}(\Delta)$ as desired. For some $\psi_{\mathrm{sub}}(x,t,\alpha_{L},\alpha_{U},\bar b)$ to be a subinterval, we must check that it actually equals $I(t,\alpha_{L},\alpha_{U})$, and that that set is not crossed by any other balls in $\mathcal{B}$. If $\bar b = (b_1,\dots,b_{p-1})$, then there are $j_1,\dots,j_{p-1}\in I$ with this set equal to $B_{\alpha_L}(t)\setminus \bigcup_{k=1}^{p-1} B_{\alpha_U}(c_{j_k}(b_k))$. This is actually $I(t,\alpha_L,\alpha_U)$ as long as there is no $i\in I$, $b \in B$ with $v(c_i(b)-t)=\alpha_U$, but $c_i(b) \not \in \bigcup_{k=1}^{p-1} B_{\alpha_U}(c_{j_k}(b_k))$. The only way for this to happen is if $v(c_i(b)-c_{j_k}(b_k))=\alpha_U$ for all $1\leq k< p$, so let $\mathcal{I}_1(\Delta)$ be the set of all $b \in B$ where this happens.

For $\Delta=I(t,\alpha_L,\alpha_U)$ to not be a subinterval, it must be crossed by some ball $B_\alpha(t^*)\in \mathcal{B}$. Such a ball crosses $I(t,\alpha_L,\alpha_U)$ if and only if $t^* \in B_{\alpha_L}(t)$, $\alpha_L<\alpha<\alpha_U$, and $$B_{\alpha}(t^*)\setminus \bigcup_{t' \in T \cap B_{\alpha_U-1}(t)}B_{\alpha_U}(t)\neq \emptyset.$$ This last condition follows from the previous two, as
$$\bigcup_{t' \in T \cap B_{\alpha_U-1}(t)}B_{\alpha_U}(t) \subsetneq B_{\alpha_U-1}(t),$$ and if $\alpha<\alpha_U$, then either $B_{\alpha_U-1}(t) \subset B_{\alpha}(t^*)$ or they are disjoint. The radius $\alpha$ can either be $v(c_j(b)-c_k(b))$, where $t^*=c_i(b)$, for some $i,j,k\in I$, or $v(t'-t^*)$ for some $t' \in T$. Let $\mathcal{I}_2(\Delta)$ be the set of all $b$ such that for some $i,j,k\in I$, $\alpha_L<v(c_j(b)-c_k(b))<\alpha_U$ and $\alpha_L<v(c_i(b)-t)$. This handles the former case. In the latter case, where $\alpha=v(t'-t^*)$, we see that as $\alpha_L<\alpha$, $t'\in B_{\alpha_L}(t^*)=B_{\alpha_L}(t)$, so $\alpha_L<v(t-t')$. Also, $\min\left\{v(t-t'),v(t-t^*)\right\}\leq v(t'-t^*)<\alpha_U$, so either the ball $B_{v(t-t')}(t)$ or $B_{v(t-t^*)}(t)$ has radius between $\alpha_L$ and $\alpha_U$, and thus crosses $\Delta$. Thus $\Delta$ is crossed by a ball of the form $B_{v(t'-t^*)}(t^*)$ if and only if it is crossed by a ball of the form $B_{v(t-t')}(t')$ if and only if there is some $t'\in T$ with $\alpha_L<v(t-t')<\alpha_U$, so we let $\mathcal{I}_3(\Delta)$ be the set of all $b$ such that there exists $i \in I$ with $\alpha_L<v(t-c_i(b))<\alpha_U$.

Then if we let $\mathcal{I}(\Delta)=\mathcal{I}_1(\Delta)\cup \mathcal{I}_2(\Delta)\cup \mathcal{I}_3(\Delta)$, which is uniformly definable from just the parameters used to define $\Delta$, then $\Delta$ is a subinterval if and only if $B\cap \mathcal{I}(\Delta)=\emptyset$, as desired.

\subsection{Counting the Distal Cell Decomposition}
To calculate the distal density of $\Phi$, we will count the number of cells of $\mathcal{T}(B)$ by following Bobkov's estimate of $\abs{S^\Phi(B)}$. Because our cells are defined less in terms of $x$ itself than the values $p_i(x)$, we define a function to shift our problem to study those values directly:
\begin{defn}[{\cite[Def. 4.3.4]{bobkov}}]
	Let $f: \Q_p^{\abs{x}}\to \Q_p^{I}$ be $(p_i(x))_{i\in I}$. Define the segment set $\mathrm{Sg}$ to be the image $f(\Q_p^{\abs{x}})$. 
\end{defn}

We will need a notation for recording certain coefficients of elements of $\Q_p$:
\begin{defn}[{\cite[Def. 4.2.9]{bobkov}}]
	For $c \in \Q_p$, $\alpha<\beta \in v(\Q_p)$, $c$ can be expressed uniquely as $\sum_{\gamma \in v(\Q_p)}c_\gamma p^\gamma$ with $c_\gamma \in \{0,1,\dots,p-1\}$. Then define $c \upharpoonright [\alpha,\beta)$ to be the tuple $(c_\alpha,c_{\alpha+1},\dots,c_{\beta-1})\in \{0,1,\dots,p-1\}^{\beta-\alpha}$.
\end{defn}

This coefficient function $\upharpoonright$ will be useful in allowing us to reduce the information of $\{a_i:i \in I\} \in \Q_p^I$ to a linearly independent subset together with a finite number of coordinates, using this lemma:
\begin{lem}[{\cite[Cor. 4.3.2]{bobkov}}]\label{5.1.2}
	Suppose we have a finite collection of vectors $\{\vec p_i\}_{i\in I}$ with each $\vec p_i \in \Q_p^{\abs{x}}$. Suppose $J \subseteq I$ and $i \in I$ satisfy $\vec p_i \in \mathrm{span} \{\vec p_j \}_{j\in J}$ , and we have $\vec c \in \Q_p , \alpha \in v(\Q_p)$ with $v(\vec p_j \cdot \vec c ) > \alpha$ for all $j \in J$. Then $v(\vec p_i \cdot \vec c ) > \alpha - \gamma$ for some $\gamma \in v(\Q_p),\gamma\geq 0$. Moreover $\gamma$ can be
chosen independently from $J, j, \vec c, \alpha$ depending only on $\{\vec p_i\}_{i\in I}$.
\end{lem}

As each homogeneous linear polynomial $p_i(x)$ can be written as the dot product $\vec p_i \cdot x$ for some $\vec p_i \in \Q_p^{\abs{x}}$, let $\gamma \in v(\Q_p)_{\geq 0}$ satisfy the criteria of Lemma \ref{5.1.2} for $\{\vec p_i\}_{i \in I}$.

\begin{defn}[{\cite[Def. 4.3.3]{bobkov}}]
	Any $a \in \Q_p$ belongs to a unique subinterval $I(t,\alpha_L,\alpha_U)$. Define $T-\mathrm{fl}(a):=\alpha_L$.
\end{defn}
Using this function, we partition $\mathrm{Sg}$ into $(2\abs{I})!$ pieces, corresponding to the possible order types of $\{T-\mathrm{fl}(x_i):i \in I\}\cup\{T-\mathrm{val}(x_i):i \in I\}$. We will show that each piece of this partition intersects only $\bigO(\abs{B}^{\abs{x}})$ cells of $\mathcal{T}(B)$.

Let $\mathrm{Sg}'$ be a piece of the partition. Relabel the functions $p_i$ such that
$$T-\mathrm{fl}(a_1)\geq \dots \geq T-\mathrm{fl}(a_{\abs{I}})$$
for all $(a_i)_{i \in I} \in \mathrm{Sg}'$. Using a greedy algorithm, find $J \subseteq I$ such that $\{\vec p_j\}_{j \in J}$, with the new labelling, is linearly independent, and for each $i \in I$, $\vec p_i$ is a linear combination of $\{\vec p_j\}_{j \in J, j<i}$.

\begin{defn}
	\begin{itemize}
		\item Denote $\{0,\dots,p-1\}^\gamma$ as Ct.
		\item Let Tp be the set of all subinterval types. Lemma 4.2.11 from \cite{bobkov} shows that $\abs{\mathrm{Tp}}\leq K$, where $K$ is a constant that does not depend on $B$.
		\item Let Sub be the set of all subintervals. Lemma 4.2.4 from \cite{bobkov} tells us that $\abs{\mathrm{Sub}}=O(\abs{B})$.
	\end{itemize}
\end{defn}

Now we can define a function identifying subintervals, subinterval types, and $\gamma$ many coefficients of the components of each element of $\mathrm{Sg}'$: 
\begin{defn}
	Define $g: \mathrm{Sg}'\to \mathrm{Tp}^I\times \mathrm{Sub}^J\times \mathrm{Ct}^{I\setminus J}$ as follows:

	Let $a = (a_i)_{i \in I}\in \mathrm{Sg}'$.

	For each $i \in I$, record the subinterval type of $a_i$ to form the component in $\mathrm{Tp}^I$.

	For each $j \in J$, record the subinterval of $a_j$ to form the component in $\mathrm{Sub}^J$.

	For each $i \in I\setminus J$, let $j\in J$ be maximal with $j <i$. Then record $a_i\upharpoonright [T-\mathrm{fl}(a_j)-\gamma,T-\mathrm{fl}(a_j)) \in \mathrm{Ct}$, and list all of these as the component in $\mathrm{Ct}^{I\setminus J}$.

	Combine these three components to form $g(a)$.
\end{defn}

As $\{\vec p_j\}_{j \in J}$ a linearly independent set in the $\abs{x}$-dimensional vector space $\Q_p^{\abs{x}}$, $\abs{J}\leq \abs{x}$, so 
$$\abs{\mathrm{Sg}'\to \mathrm{Tp}^I\times \mathrm{Sub}^J\times \mathrm{Ct}^{I\setminus J}}=\bigO(K^{\abs{I}}\cdot\abs{B}^{\abs{J}}\cdot p^{\gamma \abs{I\setminus J}})=\bigO(\abs{B}^{\abs{J}}),$$
and it suffices to show that if $a,a' \in \Q_p^{\abs{x}}$ are such that $f(a),f(a')\in \mathrm{Sg}'$, and $g(f(a))=g(f(a'))$, then $a,a'$ are in the same cell of $\mathcal{T}(B)$. That would show that the number of cells intersecting $\mathrm{Sg}'$ is at most $\abs{\mathrm{Sg}'\to \mathrm{Tp}^I\times \mathrm{Sub}^J\times \mathrm{Ct}^{I\setminus J}}=\bigO(\abs{B}^{\abs{x}})$. Then as the number of pieces in the partition is itself only dependent on $I$, the total number of cells in $\mathcal{T}(B)$ is also $\bigO(\abs{B}^{\abs{x}})$ as desired.

If $a,a'$ are such that $f(a),f(a')\in \mathrm{Sg}'$, then immediately we know that $(T-\mathrm{val}(p_i(a)))_{i \in I}$ and $(T-\mathrm{val}(p_i(a')))_{i \in I}$ have the same order type. If also $g(f(a))=g(f(a'))$, then for each $i \in I$, $p_i(a)$ and $p_i(a')$ have the same subinterval type, so it suffices to show that for each $i$, $p_i(a)$ and $p_i(a')$ are in the same subinterval. This is clearly true for $i \in J$, but we need to consider the $\mathrm{Ct}^{I\setminus J}$ component of $g$ to show that it is true for $i \in I \setminus J$. Bobkov shows this in Claim 4.3.8 and the subsequent paragraph of \cite{bobkov}. That argument is summarized here:

Fix such an $i \in I \setminus J$, and let $j\in J$ be maximal with $j <i$. By the definition of $\mathrm{Sg}'$, $T-\mathrm{fl}(a_i)\leq T-\mathrm{fl}(a_j)$ and $T-\mathrm{fl}(a_i')\leq T-\mathrm{fl}(a_j')$, and as $a_j,a_j'$ lie in the same subinterval, $T-\mathrm{fl}(a_j)=T-\mathrm{fl}(a_j')$. Claim 4.3.8 in \cite{bobkov} shows that $v(a_i-a_i')>T-\mathrm{fl}(a_j)-\gamma$. As the $\mathrm{Ct}$ components of $g(f(a))$ and $g(f(a'))$ are also the same, we know that $a_i\upharpoonright [T-\mathrm{fl}(a_j)-\gamma,T-\mathrm{fl}(a_j))=a_i'\upharpoonright [T-\mathrm{fl}(a_j')-\gamma,T-\mathrm{fl}(a_j'))$, but as $[T-\mathrm{fl}(a_j)-\gamma,T-\mathrm{fl}(a_j))=[T-\mathrm{fl}(a_j')-\gamma,T-\mathrm{fl}(a_j'))$ and $v(a_i-a_i')>T-\mathrm{fl}(a_j)-\gamma$, this tells us that even more coefficents of $a_i$ and $a_i'$ agree, so $v(a_i-a_i')>T-\mathrm{fl}(a_j)\geq \max(T-\mathrm{fl}(a_i),T-\mathrm{fl}(a_i'))$. Assume without loss of generality that $T-\mathrm{fl}(a_i)\leq T-\mathrm{fl}(a_i')$, and let the subintervals of $a_i$ and $a_i'$ be $I(t,T-\mathrm{fl}(a_i),\alpha_U)$ and $I(t',T-\mathrm{fl}(a_i'),\alpha_U')$. Then as $v(a_i-a_i')>T-\mathrm{fl}(a_i')$ and $v(t'-a_i')>T-\mathrm{fl}(a_i')$, the ultrametric inequality gives us $v(a_i-t')>T-\mathrm{fl}(a_i')$, so $a_i \in B_{T-\mathrm{fl}(a_i')}(t')$ and $a_i \in B_{T-\mathrm{fl}(a_i)}(t)$, so one ball is contained in the other. By the assumption on the radii, $B_{T-\mathrm{fl}(a_i')}(t') \subseteq B_{T-\mathrm{fl}(a_i)}(t)$. If the subintervals are distinct, they must be disjoint, in which case $B_{T-\mathrm{fl}(a_i')}(t') \subseteq B_{T-\mathrm{fl}(a_i)}(t)\setminus I(t,T-\mathrm{fl}(a_i),\alpha_U)$. However, $a_i \in B_{T-\mathrm{fl}(a_i')}(t') \cap I(t,T-\mathrm{fl}(a_i),\alpha_U)$, contradicting this. Thus the subintervals are the same.

\subsection{A Conjecture about Locally Modular Geometric Structures}
The following proposition, together with Theorem \ref{qpaffthm}, lends support to a conjecture about distal cell decompositions in locally modular geometric structures. Recall that a structure is \emph{geometric} when the $\mathrm{acl}$ operation defines a pregeometry and the structure is uniformly bounded (it eliminates the $\exists^\infty$ quantifier) \cite{geomstruct}. 

\begin{prop}
	The structure $\mathcal{M}$ with universe $\Q_p$ in the language $\mathcal{L}_{\mathrm{aff}}$ is a modular geometric structure.
\end{prop}
\begin{proof}
	To check this, it suffices to check that this structure is uniformly bounded, and that its algebraic closure operation $\mathrm{acl}$ gives rise to a modular pregeometry.

	First we check uniform boundedness. That is, we wish to show that for all partitioned $\mathcal{L}_{\mathrm{aff}}$-formulas $\varphi(x;y)$ with $\abs{x} = 1$, there is some $n \in \N$ such that for all $b \in M^{\abs{y}}$, either $|\varphi(M;b)| \leq n$ or $\varphi(M;b)$ is infinite.

	By Lemma \ref{empty_cell}, $\varphi(M,M^{\abs{y}})$ is a disjoint union of $(\abs{y}+1)$-cells of the form $\{(x,y) \in \Q_p \times D | v(a_1 (y)) \square_1 v(x - c(y)) \square_2 v(a_2(y)), x - c(y) \in \lambda Q_{m,n}\}$. Let $n_\varphi$ be the number of cells in that disjoint union. We will show that for all $b \in M^{\abs{y}}$, either $|\varphi(M;b)|\leq n_\varphi$ or $\varphi(M;b)$ is infinite. To do this, we will show that for each cell $\Delta$, defined by the formula $\psi(x;y)$, that for all $b \in M^{\abs{y}}$, either the fiber $\psi(M;b)$ is infinite, or $|\psi(M;b)|\leq 1$. Then for $b \in M^{\abs{y}}$, if the original set $\varphi(M;b)$ is finite, then each fiber $\psi(M;b)$ of the cells are finite, and thus each is at most a singleton. Thus $|\varphi(M;b)|$ is at most the number of cells $n_\varphi$.
	
	Now consider a formula $\psi(x;y)$ that defines an $(\abs{y}+1)$-cell, and the fibers of $\psi(M;b)$ for various $b \in M^{\abs{y}}$. The fibers are of the form $\{x | v(a_1(b)) \square_1 v(x - c(b)) \square_2 v(a_2(b)), x - c(b) \in \lambda Q_{m,n}\}$, and we will show that any set of that form is either empty, infinite, or the singleton $\{c(b)\}$.

	For simplicity, let us assume $c(b) = 0$. This amounts just to a translation of the set, and will not effect its size. Then assume $a \in \{x | v(a_1(b)) \square_1 v(x) \square_2 v(a_2(b)), x \in \lambda Q_{m,n}\}$, and we will show either that the set is $\{a\}$, or that it is infinite. If $\lambda = 0$, then $\lambda Q_{m,n} = \{0\}$, so we have $a = 0$ and the set is $\{0\}$. Thus we assume $\lambda \ne 0$.
	As $a \in \lambda Q_{m,n}$, there are some $k \in \Z, z \in \Z_p$ such that $a = \lambda p^{km}(1+p^nz)$, and $v(a) = v(\lambda) + km + v(1 + p^nz)$. As $n \neq 0$, we have $v(p^n z) = nv(z) \geq n > 0$, so $v(1 + p^nz) = v(1) = 0$ by the ultrametric property, and $v(a) = v(\lambda) + km$. Now for any $z' \in \Z_p$, $v(\lambda p^{km}(1+p^nz')) = v(a)$, and $\lambda p^{km}(1+p^nz') \in \lambda Q_{m,n}$, so $\lambda p^{km}(1+p^nz')$ is also in this set. As $\lambda \ne 0$, these are all distinct elements of the set, which is infinite.

	Now we check that $\mathrm{acl}$ gives rise to a modular pregeometry. To do this, it suffices to check that $\mathrm{acl}$ is just the span operation, equal to $\mathrm{acl}$ in the plain vector space language, which also gives rise to a modular pregeometry. If $B \subseteq M, a \in M$, then $a \in \mathrm{acl}(B)$ if and only if there exists a formula $\varphi(x;y)$ with $\abs{x}=1$ and a tuple $b \in B^{\abs{y}}$ such that $\varphi(M,b)$ is finite and $\mathcal{M}\models \varphi(a,b)$. If we decompose $\varphi(M;M^{\abs{y}})$ into cells, then we see that there must exist a cell (say it is defined by $\psi(x;y)$) such that $a \in \psi(M,b)$. As $\psi(M,b) \subseteq \varphi(M,b)$ is also finite, and $\psi(x;y)$ defines a cell, $\psi(M;b) = \{c(b)\}$ for a defining polynomial $c$ of the cell, which can be assumed to be linear with constant coefficient 0. Thus $a = c(b)$, so $a$ is in the span of $B$. Clearly also the span of $B$ is contained in $\mathrm{dcl}(B) \subseteq \mathrm{acl}(B)$, so $\mathrm{acl} = \mathrm{dcl}$, and both represent the span.
\end{proof}

\begin{conj}
	We conjecture that all distal locally modular geometric structures admit distal cell decompositions of exponent 1. We have already shown this in the $o$-minimal case with Theorem \ref{ominvec}, and now we have shown this for the linear reduct of $\Q_p$ with Theorem \ref{qpaffthm}.
\end{conj}

\section{$\Q_p$, the Valued Field}
\label{qpmac}

Let $\mathcal{K}$ be a $P$-minimal field, taken as a structure in Macintyre's language, which consists of the language of rings together with a symbol to define the valuation and a unary relation $P_n$ for each $n \geq 2$, interpreted so that $P_n(x)\iff \E y,y^n=x$. While the symbol to define the valuation can be chosen either to be a unary predicate defining the valuation ring or a binary relation $\mid$ interpreted so that $x|y\iff v(x)\leq v(y)$, we will refer directly to the valuation $v$ for legibility. The symbols $P_n$ are included so that this structure has quantifier-elimination \cite{mac}. Furthermore, assume that $\mathcal{K}$ has definable Skolem functions. (This assumption is only required to invoke the cell decomposition seen at equation 7.5 from \cite{vcdensityi}. The existence of this cell decomposition is shown to be equivalent to definable Skolem functions in \cite{skolem}.)

\begin{thm}\label{qpmacthm}
	Let $\Phi$ be a finite set of formulas of the form $\varphi(x;y)$. Then $\Phi$ admits a distal cell decomposition with exponent $3\abs{x}-2$.
\end{thm}
\begin{proof}
	This follows from Lemma \ref{qpmaclem} below, together with Theorem \ref{dimensioninduction}.
\end{proof}

\begin{lem}\label{qpmaclem}
	If $\abs{x}=1$, then $\Phi$ admits a distal cell decomposition $\mathcal{T}$ with 3 parameters and exponent 1.
\end{lem}

In the rest of this section, we prove Lemma \ref{qpmaclem}.

\subsection{Simplification of $\Phi$}
To construct our distal cell decomposition, we start with a simpler notion of cell decomposition. Each formula $\varphi(x;y)$ with $\abs{x} = 1$, and thus every $\varphi \in \Phi$, has a cell decomposition in the sense that $\varphi(x;y)$ is equivalent to the disjoint disjunction of the formulas $\varphi_i(x;y):1\leq i \leq N$, each of the form 
$$v(f(y))\square_{1}v(x-c(y))\square_{2}v(g(y))\wedge P_n(\lambda(x-c(y)))$$
for some $n,N >0$, where $\square_{1}$ is $<$ or no condition, $\square_{2}$ is $\leq$ or no condition, $f,g,c$ are $\emptyset$-definable functions, and $\lambda \in \Lambda$, a finite set of representatives of the cosets of $P_n^\times$. By Hensel's Lemma, we can choose $\Lambda \subset \Z \subseteq \mathrm{dcl}(\emptyset)$, so that each cell is $\emptyset$-definable \cite{mac}. Let $F$ be the set of all functions appearing as $f,g$ in these formulas, and $C$ the set of all functions appearing as $c$ (See equation 7.5, \cite{vcdensityi}).

Now we define $\Phi_{F,C,\Lambda}(x;y)$ as the set of formulas $\{v(f(y))<v(x-c(y)):f \in F,c \in C\}\cup\{P_n(\lambda(x-c(y))):c \in C,\lambda \in \Lambda\}$. It is easy to check that every formula of $\Phi$ is a boolean combination of formulas in $\Phi_{F,C,\Lambda}$, so a distal cell decomposition for $\Phi_{F,C,\Lambda}$ will also be a distal cell decomposition for $\Phi$. Thus we may assume that $\Phi$ is already of the form $\Phi_{F,C,\Lambda}$. For additional ease of notation, we also assume $F$ contains the constant function $f_0:y\mapsto 0$.

\subsection{Subintervals and subinterval types}
Let $B_r(c)$ denote again the open ball centered at $c$ with radius $r$: $B_r(c)=\{x\in K:v(x-c)> r\}$. Fix a finite set $B \subset M^{\abs{y}}$, and let $\mathcal{B}$ be a set of balls, similar to those referred to in \cite{vcdensityi}, Section 7.2 as ``special balls defined over $B$'', which we express as $\mathcal{B} := \mathcal{B}_F\cup \mathcal{B}_C$, where
$$\mathcal{B}_F=\{B_{v(f(b))}(c(b)):b \in B,f \in F,c \in C\}$$
and
$$\mathcal{B}_C=\{B_{v(c_1(b_1)-c_2(b_2))}(c_1(b_1)):b_1,b_2\in B,c_1,c_2 \in C\}.$$

Clearly $\abs{\mathcal{B}_F}=\bigO(\abs{B})$. It is less clear that $\abs{\mathcal{B}_C}=\bigO(\abs{B})$, but this is a consequence of \cite[Lemma 7.3]{vcdensityi}. Thus $\abs{\mathcal{B}}=\bigO(\abs{B})$.

\begin{defn}
We now define \emph{subintervals} and surrounding notation, analogously to Definition \ref{subinterval}, but with a different notion of subinterval types.
\begin{itemize}
 	\item Define a \emph{subinterval} as an atom in the boolean algebra generated by $\mathcal{B}$.
 	\item Each subinterval can be expressed as $I(t,\alpha_L,\alpha_U)$ where
$$I(t,\alpha_L,\alpha_U)=B_{\alpha_L}(t)\setminus \bigcup_{t' \in T\cap B_{\alpha_U-1}(t)} B_{\alpha_U}(t'),$$
for some $t=c_i(b_0)$ with $i \in I, b_0 \in B$, and $\alpha_L=\alpha_1(b_0,b_1),\alpha_U=\alpha_2(b_0,b_2)$, with $\alpha_1,\alpha_2$ chosen from a finite set $A$ of $\emptyset$-definable functions $\Q_p^2 \to \Gamma$, including two functions defined, by abuse of notation, as $\pm \infty$.
	\item The subinterval $I(t,\alpha_L,\alpha_U)$ is said to be \emph{centered} at $t$.
	\item For $a \in \Q_p$, define $T-\mathrm{val}(a):= v(a-t)$, where $a$ belongs to a subinterval centered at $t$. As in Definition \ref{Tval}, this is well-defined.
	\item Given a subinterval $I(t, \alpha_L, \alpha_U)$, two points $a_1,a_2$ in that subinterval are defined to have the same \emph{subinterval type} if one of the following conditions is satisfied:
\begin{enumerate}
	\item $\alpha_L+2v(n)< T\mathrm{-val}(a_i) < \alpha_U-2v(n)$ for $i=1,2$ and $(a_1-t)(a_2-t)^{-1}\in P_n^\times$
	\item $\neg(\alpha_L+2v(n) < T\mathrm{-val}(a_i)<\alpha_U-2v(n))$ for $i=1,2$ and $T\mathrm{-val}(a_1)=T\mathrm{-val}(a_2)< v(a_1-a_2)-2v(n)$
\end{enumerate}
\end{itemize}
\end{defn}
We will construct a distal cell decomposition $\mathcal{T}(B)$ where each cell consists of all points in a fixed subinterval with a fixed subinterval type. There are several requirements to check for this:
\begin{enumerate}
	\item The sets of points in a fixed subinterval with a fixed subinterval type are uniformly definable from three parameters in $B$.
	\item If two points lie in the same subinterval and have the same subinterval type, then they have the same $\Phi$-type over $B$.
	\item $K$ has $\bigO(\abs{B})$ subintervals, and each divides into a constant number of subinterval types.
\end{enumerate}
The first and second requirements will verify that this is a valid distal cell decomposition. The third will verify that $\abs{\mathcal{T}(B)} \leq \bigO(\abs{B})$, and thus that $\mathcal{T}$ has exponent 1. The first will guarantee that $\mathcal{T}$ uses only three parameters.

First we check the first requirement. We see that the triple $(t,\alpha_L,\alpha_U)$ can always be defined from a triple $(b_0,b_1,b_2) \in B^3$, so it suffices to show that each cell (subinterval type) in the subinterval $I(t,\alpha_L,\alpha_U)$ can be defined from $(t,\alpha_L,\alpha_U)$ and no other parameters in $B$. Note that while in Section \ref{qpaff}, we showed that the subintervals are uniformly definable, and the same argument would hold here, the defining formulas there may need more than three parameters, so we give a different argument.

A subinterval type of the first kind can be defined from $t,\alpha_L,\alpha_U$ by $\psi_\lambda(t,\alpha_L,\alpha_U):=\alpha_L + 2v(n)<v(x-t)<\alpha_U - 2v(n) \wedge P_n(\lambda(x-t))$. A subinterval type of the second kind is just a ball, of the form $B_{r+2v(n)}(q)$, where either $r = \alpha_L + i$ with $0<i\leq 2v(n)$, or $r = \alpha_U - i$, with $0 \leq i \leq 2v(n)$, and $q$ satisfies $T-\mathrm{val}(q)=r$, which is implied by $v(t-q)= r$. For a fixed $t,\alpha_L,\alpha_U$, there are a constant number of choices for $r$, and $q$ can be chosen to be $p^r (q_0)+t$, where $q_0$ is chosen from a set $Q$ of representatives for open balls of radius $2v(n)$ such that $v(q_0)=0$.

Given a potential cell $\Delta$ which represents a subinterval type within the set $I(t,\alpha_L, \alpha_U)$, we want to define $\mathcal{I}(\Delta)$ so that $\mathcal{I}(\Delta) \cap B = \emptyset$ if and only if there actualy is a subinterval $I(t,\alpha_L,\alpha_U)$. There is such an interval if and only if there are no balls in $\mathcal{B}$ strictly containing $B_{\alpha_U}(t)$ and strictly contained in $B_{\alpha_L}(t)$. A ball $B_{v(f(b)}(c(b)) \in \mathcal{B}_F$ for some $b \in B, f \in F, c \in C$ lies between those two balls if and only if $\alpha_L < v(f(b)) < \alpha_U$ and $v(f(b)) < v(t - c(b))$, so define
$$\theta_{f,c}(y;t,\alpha_L,\alpha_U) := \alpha_L < v(f(y)) < \alpha_U \wedge v(f(b)) < v(t - c(b)).$$
A ball $B_{v(c_1(b_1)-c_2(b_2))}(c_1(b_1)) \in \mathcal{B}_C$ for some $b_1, b_2 \in B, c_1, c_2 \in C$ lies between those two balls if and only if $\alpha_L < v(c_1(b_1)-c_2(b_2)) < \alpha_U$ and $v(c_1(b_1)-c_2(b_2)) < v(t - c_1(b_1))$. If this is true, then $B_{v(c_1(b_1)-c_2(b_2))}(c_1(b_1)) = B_{v(t-c_2(b_2))}(t)$, so it is enough to check if there is a ball $B_{v(t - c(b))}(t)$ that lies between those two balls. That happens if and only if $\alpha_L < v(t - c(b)) < \alpha_U$, so define
$$\theta_{c}(y;t,\alpha_L,\alpha_U) := \alpha_L < v(t - c(y)) < \alpha_U.$$ Then $\mathcal{I}(\Delta)$ is defined by the formula $$\bigvee_{c \in C}\left(\theta_c(y;t,\alpha_L,\alpha_U) \vee \left(\bigvee_{f \in F}\theta_{f, c}(y;t,\alpha_L,\alpha_U)\right)\right)$$
as desired.

Now we will check the third requirement. Ordering the balls of $\mathcal{B}$ by inclusion forms a poset, whose Hasse diagram can be interpreted as a graph. By the ultrametric property, any two intersecting balls are comparable in this ordering, which rules out cycles in the graph. As the number of vertices is $\abs{\mathcal{B}}=\bigO(\abs{B})$ and the graph is acyclic, the number of edges is also $\bigO(\abs{B})$. There are also $\bigO(\abs{B})$ subintervals, because there is (almost) a surjection from edges of the graph to subintervals: given an edge between $B_1$ and $B_2$, assuming without loss of generality that $B_2 \subsetneq B_1$, we can assign it to the subinterval $I(t, \alpha_L, \alpha_U),$ where $t \in B_2$, $\alpha_L$ is the radius of $B_1$, and $\alpha_U$ is the radius of $B_2$. This omits the subintervals with outer ball $K$, and the subintervals representing minimal balls in $\mathcal{B}$, but there are $\bigO(\abs{B})$ of those as well.

Now we will check that each subinterval breaks into only a constant number of subinterval types. Fix a subinterval $I(t, \alpha_L,\alpha_U)$. Then the subinterval types of the first kind correspond with cosets of $P_n^\times$, of which there are $n$ (or $n+1$ if one takes into account the fact that 0 is not in the multiplicative group at all). As in Section \ref{qpaff}, or \cite[Lemma 4.2.11]{bobkov}, there will also be a constant number of subinterval types of the second kind. We have seen that these can be defined as $B_{r+2v(n)}(q)$. For our fixed $(t,\alpha_L,\alpha_U)$, $r$ must be either $\alpha_L + i$ with $0<i \leq 2v(n)$ or $\alpha_U - i$ with $0 \leq i \leq 2v(n)$, which leaves only finitely many choices. For a fixed $r$, $q$ must be of the form $p^r(q_0)+t$, where $q_0$ is chosen from a fixed finite set, so there are $\abs{Q}$ choices of $q$.

Now we check the second requirement.
Let $\varphi \in \Phi, b \in B$. Then $\varphi(x;b)$ is either of the form $v(f(b))<v(x - c(b))$ for $f \in F, c \in C$ or $P_n(\lambda(x-c(b)))$ for $c \in C, \lambda \in \Lambda$.

If $\varphi(x;b)$ is $v(f(b))<v(x - c(b))$, then the set of points satisfying $\varphi(x;b)$ is a ball in $\mathcal{B}$, so a subinterval, as an atom in the boolean algebra generated by $\mathcal{B}$, is not crossed by that ball, or the formula $v(f(b))<v(x-c(b))$. Thus each cell of $\mathcal{T}(B)$, being a subset of a subinterval, is not crossed by $\varphi(x;b)$.

Now it suffices to check that each cell is not crossed by $\varphi(x;b)$, where $\varphi(x;b)$ is $P_n(\lambda(x-c(b)))$ for $c \in C, \lambda \in \Lambda$. To do this, we will need the following lemma:
\begin{lem} [7.4 in \cite{vcdensityi}] \label{subintervaltypelemma}
	Suppose $n > 1$, and let $x, y, a \in K$ with $v(y - x) > 2v(n) + v(y - a)$. Then $(x - a)(y - a)^{-1} \in P_n^\times$.
\end{lem}

We will show that any two points $a_1, a_2$ in a given subinterval $I(t,\alpha_L, \alpha_U)$ with a given subinterval type satisfy $(a_1-c(b))(a_2-c(b))^{-1} \in P_n^\times$. This shows that $\mathcal{K}\models P^n(\lambda(a_1 - c(b))) \iff P^n(\lambda(a_2 - c(b)))$, so the cell defined by points in that subinterval with that subinterval type is not crossed by $\varphi(x;b)$.

We will do casework on the two kinds of subinterval types, but for both we use the fact that the definition of $I(t,\alpha_L,\alpha_U)$ implies that either $v(t-c(b))\leq \alpha_L$, or $v(t-c(b))\geq\alpha_U$.

In the first kind of subinterval type, we have $(a_1-t)(a_2-t)^{-1} \in P_n^\times$ by definition, so it suffices to show, without loss of generality, that $(t-a_1)(c(b)-a_1)^{-1} \in P_n^\times$. Lemma \ref{subintervaltypelemma} shows that this follows from $v(t-c(b))> 2v(n)+v(t-a_1)$. As $T-\mathrm{val}(a_1) = v(t-a_1)$, this is equivalent to $v(t-c(b))\geq \alpha_U$. By the construction of $I(t,\alpha_L,\alpha_U)$, this is one of two cases, and we are left with the case $v(t-c(b))\leq \alpha_L$. In that case, $v(t-c(b))+2v(n)<v(t-a_1)$. Thus $(a_1 - c(b))(t-c(b))^{-1} \in P_n^\times$, and similarly, $(a_2-c(b))(t-c(b))^{-1} \in P_n^\times$, so we get $(a_1-c(b))(a_2-c(b))^{-1}\in P_n^\times$. 

In the second kind of subinterval type, we have $v(a_1-t)=v(a_2-t)<v(a_1-a_2)-2v(n)$. If $v(t-c(b))\geq \alpha_U$, then as $a_1 \in I(t,\alpha_L,\alpha_U)$, we have $\alpha_L < v(a_1 -t)\leq \alpha_U$, we have $v(a_1-c(b))= v(a_1-t)$ by the ultrametric property. Thus $v(a_1-c(b))+2v(n)<v(a_1-a_2)$, so by Lemma \ref{subintervaltypelemma}, $(a_1-c(b))(a_2-c(b))^{-1} \in P_n^\times$. In the other case, $v(t-c(b))\leq \alpha_L< v(a_1-t)$, so the lemma tells us that $v(a_1-c(b)) = v(t-c(b))<v(a_1-t)-2v(n)$, so by the lemma, $v(a_1-c(b))(a_1-t)^{-1} \in P_n^\times$, and also $v(a_2-c(b))(a_2-t)^{-1} \in P_n^\times$, so as also $v(a_1-t)+2v(n)<v(a_1-v_2)$, so $(a_1-t)(a_2-t)^{-1} \in P_n^\times$, so we can combine all these facts to get $(a_1-c(b))(a_2-c(b))^{-1} \in P_n^\times$.

\section{Zarankiewicz's Problem}\label{zarsec}
In this section, we introduce background on Zarankiewicz's problem, and the bounds known for the case of distal-definable bipartite graphs in general.
We then combine these general bounds with the bounds on distal cell decompositions throughout in this paper, arriving at concrete combinatorial corollaries for the distal structures we have discussed.

\subsection{Background}\label{zar_background}
First we will want to define the notion of a bigraph.
A bigraph consists of a pair of sets $X,Y$ and a relation $E \subset X \times Y$ such that $E$ is a bipartite graph with parts $X$ and $Y$.
We say that such a bigraph \emph{contains} a $K_{s,u}$ if there is a subset $A \subset X$ with $\abs{A}=s$ and a subset $B \subset Y$ with $\abs{B}=t$ such that $E$ restricted to $A \times B$ is a complete bipartite graph (isomorphic to $K_{s,u}$).

Zarankiewicz's problem asks to bound asymptotically in $m$ and $n$ the number of edges in the largest bipartite graph on $m \times n$ omitting the subgraph $K_{s,t}$.
Better bounds are known when we fix a particular infinite bigraph $E$ omitting some $K_{s,t}$, and bound the size of the largest subgraph with parts of size $m,n$ respectively.
If $P,Q$ are subsets of the parts of $E$, then we write $E(P,Q)$ to denote the set of edges between $P,Q$, so we concern ourselves with bounding $|E(P,Q)|$ in terms of $|P|$ and $|Q|$.
This applies easily to problems in incidence geometry - if $\Gamma$ is a family of curves on $\R^n$,
we may consider an incidence graph on parts $\R^n$ and $\Gamma$ defined by placing an edge between $(p,\gamma)$ exactly when $p \in \gamma$.
When these curves are algebraic of bounded degree, B\'ezout's theorem bounds the size of a complete bipartite subgraph $K_{s,t}$ in this incidence graph,
and then we are interested in the number of edges (incidences) between a finite set of points and a finite set of curves.
For a general reference on incidence geometry, see \cite{sheffer_book}.

We will concern ourselves with the case where the bigraph is definable in a distal structure.
In the incidence example, this happens when the curves in $\Gamma$ are uniformly definable in some distal structure on $\R$.
In \cite{cgs}, the authors set an upper bound for Zarankiewicz's problem in bigraphs definable in a distal structure,
using distal cell decompositions as the foundation of their approach.
The resulting bound depends essentially on the distal density of the definable graph - this is our primary motivation for defining distal density and distal exponents in this paper.

The approach of \cite{cgs} follows a classic divide-and-conquer argument used in \cite[Section 4.5]{matousek_GTM} to prove the Szemer\'edi-Trotter theorem,
which states that if we let $\Gamma$ be the set of lines in $\R^2$, then
$$|E(P,Q)| = \bigO\left(|P|^{2/3}|Q|^{2/3} + |P| + |Q|\right).$$
This is proven using cuttings:
\begin{defn}
	Let $\mathcal{F}$ be a finite family of subsets of a set $X$ with $|\mathcal{F}| = n$.
	Given a real $1 < r < n$, we say that a family $\mathcal{C}$ of subsets of $X$ is a $\frac{1}{r}$-\emph{cutting for} $\mathcal{F}$ when
	$\mathcal{C}$ forms a cover of $X$ and each set $C \in \mathcal{C}$ is crossed by at most $\frac{n}{r}$ elements of $\mathcal{F}$.
\end{defn}
Cuttings differ from abstract cell decompositions in that a limited amount of crossing is allowed, but they are still related.
In \cite[Section 6.5]{matousek_GTM}, a bound (\cite[Lemma 4.5.3]{matousek_GTM}) is given on the size of an $\frac{1}{r}$-cutting into triangles with respect to any finite set of lines.
For a given set of points and a given set of lines, a particular value of $r$ is chosen, an $\frac{1}{r}$-cutting is found,
and then for each triangle in the cutting, the set of incidences between points in the triangle and lines that cross the triangle is bounded.
These bounds are summed, and after considering some exceptional cases, this proves Szemer\'edi-Trotter.

In \cite{cgs}, meanwhile, the authors find uniformly definable cuttings for each definable relation, starting with a distal cell decomposition.
The size of the cutting given by this cutting lemma scales directly with the size of the given distal cell decomposition,
so the bounds on distal cell decompositions throughout this paper also function as bounds on the sizes of cuttings.
\begin{fact}[{Distal Cutting Lemma: \cite[Theorem 3.2]{cgs}}]
	Let $\phi(x;y)$ be a formula admiting a distal cell decomposition of exponent $d$.
	Then for any natural $n$ and any real $1 < r < n$, there exists $t = \bigO(r^d)$ such that for any finite $H \subseteq M^{|y|}$ of size $n$,
	there are uniformly definable sets $X_1, \dots, X_t \subseteq M^{|x|}$ which form an $\frac{1}{r}$-\emph{cutting for } $\{\phi(x;h): h \in H\}$.
\end{fact}
The proof of this also follows the proof of the cutting lemma for lines in \cite[Sections 4.6 and 6.5]{matousek_GTM},
which in turn uses the random sampling technique of Clarkson and Shor.\cite{clarkson_shor}.

From this cutting lemma, a similar divide-and-conquer argument works.
Given a formula $\phi(x;y)$ on a distal structure $M$ defining a bigraph $E$ on $M^{|x|} \times M^{|y|}$,
for any finite subset $H \subseteq M^{|y|}$,
the authors of \cite{cgs} use a distal cell decomposition and the distal cutting lemma to find a suitable cutting for $\{\phi(x;h): h \in H\}$.
They then, in summary, use other tools to bound the incidences between the points in each cell of the cutting and formulas $\phi(x;h)$ which cross it,
and combine these bounds to find a final result, quoted here in our terminology:
\begin{fact}[{\cite[Theorem 5.7]{cgs}}]\label{zarcgs}
	Let $\mathcal{M}$ be a structure and $d,t \in \N_{\geq 2}$. Assume that $E(x,y)\subseteq M^{\abs{x}}\times M^{\abs{y}}$ is a definable relation given by an instance of a formula $\theta(x,y;z)\in\mathcal{L}$, such that the formula $\theta'(x;y,z):=\theta(x,y;z)$ has a distal cell decomposition of exponent $t$, and such that the VC density of $\theta''(x,z;y):=\theta(x,y;z)$ is at most $d$. Then for any $k \in \N$ there is a constant $\alpha=\alpha(\theta,k)$ satisfying the following.

	For any finite $P\subseteq M^{\abs{x}},Q \subseteq M^{\abs{y}}$, $\abs{P}=m,\abs{Q}=n$, if $E(P,Q)$ is $K_{k,k}$-free, then we have:
	$$\abs{E(P,Q)}\leq \alpha\left(m^{\frac{(t-1)d}{td-1}}n^{\frac{t(d-1)}{td-1}}+m+n\right).$$
\end{fact}
While $d,t$ are assumed to be integral in their theorem statement, they could be replaced with any real $d,t \in \R_{\geq 2}$ and their proof would work unchanged. If $\theta'$ has distal density $t$, then it is not known if $\theta$ must have a distal cell decomposition of exponent precisely $t$. However, we can still get nearly the same bound, as for all $\varepsilon>0$, $\theta'$ has a distal cell decomposition with exponent $t+\varepsilon$. As $\lim_{\varepsilon \to 0}\frac{(t+\varepsilon-1)d}{(t+\varepsilon)d-1}=\frac{(t-1)d}{td-1}$, and $\frac{(t+\varepsilon)(d-1)}{(t+\varepsilon)d-1}\leq \frac{t(d-1)}{td-1}$, the theorem still holds for $\theta'$ with distal density $t$, except with the final bound replaced by
$$\abs{E(P,Q)}\leq \alpha\left(m^{\frac{(t-1)d}{td-1}+\varepsilon}n^{\frac{t(d-1)}{td-1}}+m+n\right)$$
for arbitrary $\varepsilon > 0$ and $\alpha = \alpha (\theta, k, \varepsilon)$.

Contrast this result to an analogous result for semi-algebraic sets, using polynomial partitioning for the divide-and-conquer argument instead of cuttings:
\begin{fact}[{\cite[Corollary 1.7]{walsh}}]\label{walsh}
	Let $P$ be a set of $m$ points and let $\mathcal{V}$ be a set of $n$ constant-degree algebraic varieties, both in $\R^d$, such that the incidence graph of $P\times \mathcal{V}$ does not contain $K_{s,t}$.  Then for every $\varepsilon >0$, we have
	$$I(P,\mathcal{V})=\bigO_{d,s,t,\varepsilon}\left(m^{\frac{(d-1)s}{ds-1}+\varepsilon}n^{\frac{d(s-1)}{ds-1}}+m+n\right).$$
\end{fact}
The initial version of this result, \cite[Theorem 1.2]{sheff}, had an extra factor of $m^\varepsilon$ in the first term.
The $m^\varepsilon$ was removed first in special cases, such as in \cite[Theorem 1.5]{basu_sombra}, with a more involved application of polynomial partitioning,
eventually leading to \cite{walsh}.
\begin{rmk}
	The special case of $d = s = 2$ is proven in \cite[Theorem 1.1]{sheff}, without the extra factor of $m^\varepsilon$, using the cutting lemma strategy generalized by \cite{cgs}.
	This method would imply the rest of Fact \ref{walsh} given a distal cell decomposition of exponent $\abs{x}$ for each finite set $\Phi(x;y)$ of formulas in the language of ordered rings over $\R$.
\end{rmk}

As a last remark before examining these combinatorial applications in specific structures, we mention some other combinatorial applications of distal cell decompositions which may be improved using specific bounds like those in this paper.
While the papers are different in strategy and scope, both \cite[Theorem 2.6]{basu_seh} and \cite[Theorem 1.9]{distal_reg} apply techniques that we now recognize as distal cell decompositions and distal cutting lemmas Ramsey-theoretically,
showing that sets definable in distal structures satisfy a property that \cite{distal_reg} dubs the \emph{strong Erd\H{o}s-Hajnal property}.
The constants in this asymptotic bound are improved by providing better bounds on exponents of distal cell decompositions.

\subsection{New Results in Specific Structures}\label{zar_new}
In this subsection, we collect the results from earlier in the paper and combine them with the Zarankiewicz bounds of \cite{cgs} as cited above.

We begin by just applying Fact \ref{zarcgs} with known distal exponent and VC density bounds, listing the exponents in the resulting Zarankiewicz bounds in a table.
\begin{cor}\label{cor_table}
	Let $\mathcal{M}$ be a structure from the left column of the following table and let $E \subseteq M^{a} \times M^{b}$ be a definable bigraph.
	Then for any $k \in \N$, there is a constant $\alpha=\alpha(\theta,k)$ such that for the corresponding values of $q$ and $r$ in this table,
	and any finite $P\subseteq M^{a},Q \subseteq M^{b}$, $\abs{P}=m,\abs{Q}=n$, if $E(P,Q)$ is $K_{k,k}$-free, then
	$\abs{E(P,Q)}\leq \alpha\left(m^{q}n^{r}+m+n\right).$
	\begin{figure}[H]
		\center
		\def\arraystretch{1.8}
		\begin{tabular}{|c | c | c |}
			\hline
			 $\mathcal{M}$& $q$ & $r$ \\ \hline
			 $o$-minimal expansions of groups  & $\frac{(2a-3)b}{(2a-2)b-1}$& $\frac{(2a-2)(b -1)}{(2a-2)b-1}$\\\hline
			 weakly $o$-minimal structures & $\frac{(2a-2)b}{(2a-1)b-1}$& $\frac{(2a-1)(b -1)}{(2a-1)b-1}$\\\hline
			 ordered vector spaces over ordered division rings & $\frac{(a-1)b}{ab-1}$& $\frac{a(b -1)}{ab-1}$\\\hline
			 Presburger arithmetic & $\frac{(a-1)b}{ab-1}$& $\frac{a(b -1)}{ab-1}$\\\hline
			 $\Q_p$ the valued field & $\frac{(3a-3)(2b-1)}{(3a-2)(2b-1)-1}$& $\frac{(3a-2)(2b-2)}{(3a-2)(2b-1)-1}$\\\hline
			 $\Q_p$ in the linear reduct & $\frac{(a-1)b}{ab-1}$& $\frac{a(b -1)}{ab-1}$\\\hline
		\end{tabular}
		\end{figure}
\end{cor}
\begin{proof}
	The bounds on VC densities and exponents of distal cell decompositions are listed in Theorem \ref{thm_table}.
	The VC densities come from the literature cited in that theorem, as does the exponent for the distal cell decomposition in the case of $o$-minimal expansions of groups with $a = 2$ from \cite{cgs},
	but the rest of the distal cell decomposition bounds are new to this article.
\end{proof}

In some applications to Zarankiewicz's problem, the omitted bipartite graph $K_{s,u}$ may give a better bound on the relevant VC density than is known for general formulas.
The following lemma bounds the VC density for formulas defining relations which do not contain a $K_{s,u}$:
\begin{lem}
	Let $\mathcal{M}$ be a first-order structure, and $\varphi(x;y)$ be a formula such that the bigraph with edge relation $\varphi(M^{\abs{x}};M^{\abs{y}})$ does not contain $K_{s,u}$. Then $\mathrm{vc}(\varphi)\leq s$.
\end{lem}
\begin{proof}
	An equivalent way (see \cite{vcdensityi}) of defining $\pi_\varphi(n)$ is as 
	$\max_{A \subset M^{\abs{x}},\abs{A}=n}\abs{\varphi \cap A},$
	where $\varphi \cap A$ is shorthand for $\{A \cap \varphi(M^{\abs{x}},b):b \in M^{\abs{y}}\}$.

	Given $A \subset M^{\abs{x}}$, find $B \subset M^{\abs{y}}$ such that for each subset $A_0 \in \varphi \cap A$, there is exactly one $b\in B$ such that $A_0=A\cap \varphi(M^{\abs{x}},b)$. Thus $\abs{B}=\abs{\varphi \cap A}$. 

	The number of subsets of $A$ in $\varphi \cap A$ of size less than $B$ is trivially bounded by $\sum_{i=0}^{s-1}{\abs{A} \choose i}=\bigO(\abs{A}^{s-1})$. Thus there are $\bigO(\abs{A}^{s-1})$ elements $b \in B$ for which $\abs{\varphi(M^{\abs{x}},b)\cap A}<s$. However, by assumption, for each subset $A_B\subseteq A$ of size $B$, there are most $t-1$ elements $b$ of $B$ with $\mathcal{M}\models \varphi(a,b)$ for all $a \in A_s$. Thus there are at most $(t-1){\abs{A}\choose s}=\bigO(\abs{A}^s)$ elements $b \in B$ for which $\abs{\varphi(M^{\abs{x}},b)\cap A}\geq s$, and in general, $\abs{B}=\bigO(\abs{A}^s)$, so $\pi_\varphi(n)=\bigO(n^s)$, and $\mathrm{vc}(\varphi)\leq s$.
\end{proof}

Combining this lemma with Theorem \ref{zarcgs} gives us the following Zarankiewicz bound for bigraphs defined in distal structures,
making use only of the omitted complete bipartite subgraph for the VC density bound.
\begin{cor}\label{zarcor}
	Let $\mathcal{M}$ be a structure and $t \in \R_{\geq 2}$. Assume that $E(x,y)\subseteq M^{\abs{x}}\times M^{\abs{y}}$ is a definable relation given by an instance of a formula $\theta(x,y;z)\in\mathcal{L}$, such that the formula $\theta'(x;y,z):=\theta(x,y;z)$ has a distal cell decomposition of exponent $t$, and the graph $E(x,y)$ does not contain $K_{s,u}$. Then there is a constant $\alpha=\alpha(\theta,s,u)$ satisfying the following.

	For any finite $P\subseteq M^{\abs{x}},Q \subseteq M^{\abs{y}}$, $\abs{P}=m,\abs{Q}=n$, we have:
	$$\abs{E(P,Q)}\leq \alpha\left(m^{\frac{(t-1)s}{ts-1}}n^{\frac{t(s-1)}{ts-1}}+m+n\right).$$
\end{cor}

This Corollary recalls one version of Theorem 2.6 of \cite{es}, which provides the same bound on $\abs{E(P,Q)}$ from a slightly different assumption on $t$, and either the same condition of $\varphi(x;y)$ omitting $K_{s,u}$ for some $u$, or $\varphi(x;y)$ omitting $K_{u,u}$ for some $u$ and having dual VC density at most $s$.

To phrase this corollary in terms of distal density, we must add a small error term again.
If instead $t$ is the distal density of $\theta'$, then for all $\varepsilon \in \R_{>0}$, we get the bound 
$$\abs{E(P,Q)}\leq \alpha\left(m^{\frac{(t-1)s}{ts-1}+\varepsilon}n^{\frac{t(s-1)}{ts-1}}+m+n\right),$$
where $\alpha$ depends also on $\varepsilon$.

To illustrate the generality of Corollary \ref{zarcor}, we will apply it to some specific structures. Let us first apply it to $\mathcal{M} = \R_{\mathrm{exp}} = \langle\R; 0,1,+,*,<,e^x\rangle$. This structure is an expansion of a field, and $o$-minimal by \cite{wilkie}, allowing us to apply the distal exponent bounds from Theorem \ref{ominthm}. We define an \emph{exponential polynomial} to be a function $\R^n \to \R$ in $\Z[x_1,\dots,x_n,e^{x_1},\dots,e^{x_n}]$ as in \cite{exppoly}, and an \emph{exponential-polynomial inequality} to be an inequality of exponential polynomials. As any exponential polynomial function over $\R$ is definable in this structure, a boolean combination of exponential-polynomial inequalities or equations will be as well. Combining all of this with Corollary \ref{zarcor} gives the following result:
\begin{cor} \label{cor:Rexp}
	Assume that $E(x,y)\subseteq \R^{\abs{x}}\times \R^{\abs{y}}$ is a relation given by a boolean combination of exponential-polynomial (in)equalities, and the graph $E(x,y)$ does not contain $K_{s,u}$. Then there is a constant $\alpha=\alpha(\theta,s,u)$ satisfying the following.z

	For any finite $P\subseteq \R^{\abs{x}},Q \subseteq \R^{\abs{y}}$, $\abs{P}=m,\abs{Q}=n$, we have:
	$$\abs{E(P,Q)}\leq \alpha\left(m^{\frac{(2\abs{x}-2)s}{(2\abs{x}-1)s-1}}n^{\frac{(2\abs{x}-1)(s-1)}{(2\abs{x}-1)s-1}+\varepsilon}+m+n\right).$$
\end{cor}

Let us also apply Corollary \ref{zarcor} to subanalytic sets over $\Z_p$, defined as in \cite{denef_vdd}:
\begin{defn}\label{def_subanalytic}
	\begin{itemize}
		\item A set $S \subseteq \Z_p^n$ is semianalytic if for every $x \in S$, there is an open neighborhood $U$ of $x$ such that $U \cap S$ can be defined by a boolean combination of inequalities of analytic functions.
		\item A set $S \subseteq \Z_p^n$ is subanalytic if for every $x \in S$, there is an open neighborhood $U$ of $x$ and a semianalytic set $S'$ in $U \times \Z_p^N$ for some $N$ such that $U \cap S = \pi(S')$, where $\pi : U \times \Z_p^N \to U$ is the projection map.
	\end{itemize}
\end{defn}

For any $n$, the subanalytic subsets of $\Z_p^n$ are exactly the quantifier-free definable subsets in a structure $\mathcal{R}_\mathrm{an}$, which is a substructure of the structure $\mathcal{K}_\mathrm{an}$, consisting of $\Q_p$ with its analytic structure, as described in \cite{subanalytic}. As per Theorem A'/B from \cite{subanalytic}, this structure is $P$-minimal with definable Skolem functions, we can apply the distal exponent bounds from Theorem \ref{qpmacthm}, giving us this corollary:
\begin{cor} \label{cor:Qpan}
	Assume that $E(x,y)\subseteq \Z_p^{\abs{x}}\times \Z_p^{\abs{y}}$ is a subanalytic relation, and the graph $E(x,y)$ does not contain $K_{s,u}$. Then there is a constant $\alpha=\alpha(\theta,s,u)$ satisfying the following.

	For any finite $P\subseteq \Z_p^{\abs{x}},Q \subseteq \Z_p^{\abs{y}}$, $\abs{P}=m,\abs{Q}=n$, we have:
	$$\abs{E(P,Q)}\leq \alpha\left(m^{\frac{(3\abs{x}-3)s}{(3\abs{x}-2)s-1}}n^{\frac{(3\abs{x}-2)(s-1)}{(3\abs{x}-2)s-1}+\varepsilon}+m+n\right).$$
\end{cor}

\bibliographystyle{plainurl}
\bibliography{ref.bib}

\end{document}